\documentclass{amsart}
\allowdisplaybreaks
\usepackage{amsmath,amsfonts,amssymb}%
\usepackage{exscale}%
\usepackage{amsfonts}%
\usepackage{palatino,dsfont}
\usepackage{amsthm}
\usepackage{epsfig}
\usepackage{xcolor}
\usepackage{hyperref}
\usepackage[utf8]{inputenc}
\usepackage[english]{babel}
\usepackage{mathtools}
\usepackage{cite}
\usepackage{cleveref}
\usepackage[footnote,draft,silent,nomargin]{fixme}

\hypersetup{%
	plainpages=false,%
	pdfauthor={Author(s)},%
	pdftitle={Title},%
	pdfsubject={Subject},%
	bookmarksnumbered=true,%
	colorlinks,%
	citecolor=black,%
	filecolor=black,%
	linkcolor= black,% you should probably change this to black before printing
	urlcolor=black,%
	pdfstartview=FitH%
}

\let\epsilon\varepsilon

\let\phi\varphi

\let\theta\vartheta

\newtheorem{mytheorem}{Theorem}[section]
\newtheorem{myprop}[mytheorem]{Proposition}
\newtheorem{mycor}[mytheorem]{Corollary} 
\theoremstyle{definition}
\newtheorem{mydef}[mytheorem]{Definition}

\newtheorem{myre}[mytheorem]{Remark}
\newtheorem{mylemma}[mytheorem]{Lemma}

\newcommand{\R}{\mathbb{R}}

\newcommand{\N}{\mathbb{N}}

\newcommand{\cQ}{\mathcal{Q}}
\newcommand{\cD}{\mathcal{D}}
\newcommand{\cC}{\mathcal{C}}
\newcommand{\cR}{\mathcal{R}}
\newcommand{\cS}{\mathcal{S}}
\newcommand{\cM}{\mathcal{M}}
\newcommand{\cW}{\mathcal{W}}

\newcommand{\va}{\vec{a}}

\newcommand{\vp}{\vec{p}}

\newcommand{\cn}{{n}}
\newcommand{\ca}{{a}}

\newcommand{\Span}{\operatorname{span}}
 \newcommand{\bN}{\mathbb{N}}

 \newcommand{\cP}{\mathcal{P}}

\newcommand{\bR}{\mathbb{R}}
\newcommand{\bZ}{\mathbb{Z}}

\newcommand{\supp}{\operatorname{supp}}

\newcommand{\ptt}{\mathbb{I}}
\newcommand{\brac}[1]{\langle #1\rangle_{\vec{a}}}

\newcommand{\ve}{\varepsilon}
\renewcommand{\epsilon}{\varepsilon}

\DeclareMathOperator{\sgn}{sgn}
\numberwithin{equation}{section}
\begin{document}
\title[Stable decomposition of Triebel-Lizorkin type spaces ]{Stable decomposition of homogeneous Mixed-norm Triebel-Lizorkin spaces } \author{Morten Nielsen }   \date{\today}
\begin{abstract}
We construct smooth localized orthonormal bases compatible with homogeneous mixed-norm  Triebel-Lizorkin spaces in an anisotropic setting on $\bR^d$. The construction is based on  tensor products of so-called univariate brushlet functions that are constructed using local trigonometric bases in the frequency domain. It is shown that the associated decomposition system form  unconditional bases for the homogeneous mixed-norm Triebel-Lizorkin spaces. 

In the second part of the paper we study nonlinear $m$-term nonlinear approximation with  the constructed basis in the mixed-norm setting, where the behaviour, in general, for $d\geq 2$, is shown to be fundamentally different from the unmixed case. However, Jackson and Bernstein inequalities for $m$-term approximation can still be derived. 
\end{abstract}
\maketitle

\section{Introduction}

The notion of sparse stable expansions of functions plays a fundamental role for many applications of harmonic analysis since it allows for discretization of various problems making them suitable for numerical analysis and general mathematical modeling, one example being wavelet expansions of $L_p$-functions, or of functions with some prescribed smoothness such as measured on the Besov space scale. The success of this approach is well-known  with a corresponding  diverse spectrum of algorithms for e.g.\ signal compression.

From the point of view of mathematical modeling, it is desirable to have as flexible tools as possible as it allows one to incorporate and  capture finer properties of various natural phenomena in the model. 
Recently, function spaces in anisotropic and mixed-norm settings have attached considerable interest, see for example \cite{MR3707993,MR3388786,MR3377120,MR2720206,MR2319603,MR2186983} and reference therein. This is in part driven by advances in the study of partial and pseudodifferential operators, where there is a natural desire to be able to better model and analyse anisotropic phenomena.

In this paper, we focus on stable expansions in 
 Besov and Triebel-Lizorkin spaces in an anisotropic setting.  The  Besov and Triebel-Lizorkin spaces form a  family of smoothness spaces defined on $\bR^d$ that include the Hardy and Lebesgue spaces as special cases. 
 
Recently, homogeneous Besov and Triebel–Lizorkin spaces in a mixed-norm anisotropic setting were introduced in \cite{MR3981281}, along with an adapted notion of a 
$\phi$-transform  that provides a universal, but redundant, decomposition of the various Besov and 
Triebel–Lizorkin spaces. This follows parallel to the original construction by Frazier and Jawerth, where they  successfully applied the $\phi$-transform to an in-depth study of some of the finer properties of the properties of isotropic  Besov spaces and
Triebel–Lizorkin space. Later the $\phi$-transform techniques were refined 
to obtain orthonormal wavelets, see \cite{MR1228209,MR1107300}.

The interest in mixed-norm smoothness spaces is not new. The Russian school investigated mixed-norm Besov, Sobolev and Bessel potential spaces, even in the anisotropic setting,  more than four decades ago \cite{MR521808,MR519341}.

The main contribution of the present paper is to offer a
construction of a universal orthonormal basis for $L_2(\bR^d)$, once the anisotropy has been fixed, that extends to  unconditional basis for all mixed-norm Besov and Triebel–Lizorkin spaces. The fact that the representation system is non-redundant allows for a more precise analysis of the approximation properties of the system. For example, it will be shown in Section \ref{sec:nl} that is possible to derive an inverse estimate of Bernstein type for the system. Obtaining a Bernstein estimate is a longstanding open problem for the redundant $\phi$-transform, even in the unmixed case, cf.\ \cite{MR1794807}.  

We believe that  our construction is the first
example of a universal non-redundant representation system for mixed-norm Besov and Triebel–Lizorkin spaces, except for unmixed cases associated with a lattice preserving dilation, where orthonormal $r$-regular wavelet bases have been constructed to provide such  expansion systems, see  \cite{bownik_construction_2001,calogero_wavelets_1999}.

The orthonormal basis is constructed in Section \ref{sec:onb} using a carefully calibrated
tensor product approach based on so-called univariate brushlet
systems.  Brushlets are the image of a local trigonometric basis
under the Fourier transform, and such systems were introduced in 
\cite{Laeng1990}. Brushlets have been used as a tool for image compression, see \cite{Meyer1997a}.  Univariate brushlet bases are extremely flexible and can be adapted to any type of exponential covering of the frequency axis, where e.g.\ wavelets are restricted to only dyadic decompositions. We will use this flexibility in each variable separately in the the tensor product construction presented in Section \ref{sec:onb} in order to obtain compatibility with the overall anisotropic structure on $\bR^d$. 
We give a full characterization of the mixed-norm Besov and Triebel–Lizorkin spaces in Section \ref{sec:5}, using the constructed orthonormal basis.

A relevant question that must be considered is whether the mixed-norm spaces really brings anything new to the table. In Section \ref{sec:nl}   we address this question by studying nonlinear $m$-term nonlinear approximation with  the constructed basis in the mixed-norm setting. Non-linear approximation with wavelets in the unmixed Triebel-Lizorkin setting has previously been studied in detail, see e.g.\  \cite{MR1175690,MR1866250}. Non-linear approximation in a discrete Besov and Triebel-Lizorkin space setting has been studied in \cite{Garrigos2004}.

In Section \ref{sec:nl}, we show  that the approximation behaviour in the mixed-case, in general, for $d\geq 2$, is fundamentally different from the unmixed case. In particular, it is shown that the constructed orthonormal basis cannot, in general, be normalized to form a greedy basis for the mixed-norm Triebel-Lizorkin spaces. Despite the lack of the greedy property, it is  shown in Section \ref{sec:nl} that Jackson and Bernstein inequalities for nonlinear $m$-term approximation can be derived. However, as it turn out, the inequalities have non-matching exponents, so a complete characterization of the approximation spaces for $m$-term apprioximation is not possible using the standard machinery of nonlinear $m$-term approximation.  
%and we  also 
%identify the $\alpha$-modulation spaces
%as approximation spaces associated with nonlinear $m$-term
%approximation. 

\section{The anisotropic setup and some preliminaries}
Let us first introduce the anisotropic structure on $\bR^d$ that will be used for defining the mixed-norm Besov and Triebel-Lizorkin spaces at the end of this section. 
Let $b,x\in\mathbb{R}^d$ and $t>0.$ We denote by $t^{b}x:=(t^{b_1}x_1,\dots,t^{b_d}x_d).$ We fix a vector $\vec{a}\in[1,\infty)^d$ and we introduce the anisotropic quasi-norm $|\cdot|_{\vec{a}}$ as follows: We set $|0|_{\vec{a}}:=0$ and for $x\neq0$ we set $|x|_{\vec{a}}:=t_0,$ there $t_0$ is the unique positive number such that $|t_0^{-\vec{a}}x|=1.$ One observes immediately that
\begin{equation}\label{ad1}
|t^{\vec{a}}x|_{\vec{a}}=t|x|_{\vec{a}},\;\;\text{for every}\;\;x\in\mathbb{R}^d,\;t>0.
\end{equation}
From this it follows that $|\cdot|_{\vec{a}}$ is not a norm unless  $\vec{a}=(1,\dots,1)$, in which case it is equivalent with the Euclidean norm $|\cdot|$.

The anisotropic distance can be directly compared to the Euclidean norm, see e.g.\ \cite{Bagby:1975uc,Bownik2006}, in the sense that there are  constants $c_1,c_2>0$ such that for every $x\in\mathbb{R}^d$, 
% \begin{equation}\label{ad6}
% c_1^{-1}|x|^{1/a_M}\le|x|_{\vec{a}}\le c_1|x|^{1/a_m},\;\;\text{when}\;|x|_{\vec{a}}\ge1,
% \end{equation}
% \begin{equation}\label{ad7}
% c_1^{-1}|x|^{1/a_m}\le|x|_{\vec{a}}\le c_1|x|^{1/a_M},\;\;\text{when}\;|x|_{\vec{a}}\le1,
% \end{equation}

% We shall need several times to estimate the anisotropic distance in terms of  the Euclidean distance on the full space. From (\ref{ad6}) and (\ref{ad7}) we conclude that there are constants $c_2,c_3>0$ such that for every $x\in\mathbb{R}^d,$
\begin{equation}\label{ad8}
c_1(1+|x|_{\vec{a}})^{a_m}\le1+|x|\le c_2(1+|x|_{\vec{a}})^{a_M},
\end{equation}
where we denoted $a_m:=\min_{1\le j\le d}a_j,\;a_M:=\max_{1\le j\le d}a_j.$
%For every coordinate $x_i$, $i=1,\dots,n$ it holds that
%\begin{equation}\label{maxxi}
%|x_i|^{1/a_i}\le|x|_{\vec{a}}.
%\end{equation}

We will also need the following anisotropic bracket. We consider $(1,\vec{a})\in\bR^{d+1}$ and define
$$\brac{x}:=|(1,x)|_{(1,\vec{a})},\qquad x\in\bR^d.$$ 
This quantity has been studied in detail in \cite{Borup2008,Stein1978}.
It holds that there are constants $c_3,c_4>0$ such that 
\begin{equation}\label{bracket}
  c_3\brac{x}\leq 1+|x|_{\vec{a}}\leq c_4\brac{x},\qquad x\in\bR^d.
  \end{equation}
One can  show that there is a constant $c_5>0$ such that
\begin{equation}\label{brac_esti}
  \brac{x+y}\leq c_5 \brac{x} \brac{y},\qquad x,y\in\bR^d.
  \end{equation}
  Furthermore, we define the homogeneous dimension by 
\begin{equation}\label{nu}
  \nu:=|\vec{a}|:=a_1+\cdots+a_d,
  \end{equation}
which will play an important role later when we estimate certain maximal functions. Also,  by going to polar coordinates one can deduce that for $\tau>\nu$,
\begin{equation}\label{brac_int}
	\int_{\bR^d} \brac{x}^{-\tau}dx\le c_{\tau}<\infty.
\end{equation} 

We will also need balls adapted to the anisotrophy $\va$. Let $x\in\bR^d$ and $r>0$. We denote by $B_{\vec{a}}(x,r):=\{y\in\bR^d:\;|x-y|_{\vec{a}}<r\},$ the ball of radius $r$, centred at $x$. Note that $B_{\vec{a}}(x,r)$ is convex and $|B_{\vec{a}}(x,r)|=|B_{\vec{a}}(0,1)|r^\nu$.

\subsection{Maximal operators and mixed-norm Lebesgue spaces.}\label{sec:multi} Maximal operator estimates will be of fundamental importance for  our study of Triebel-Lizorkin spaces. Let $1\leq k\leq d$. We define
\begin{equation}\label{MK}
M_k f(x)=\sup\limits_{I\in I_x^k} \dfrac{1}{|I|} \int_I |f(x_1,\dots,y_k,\dots,x_d)| dy_k,\;\;f\in L^{1}_{loc}(\bR^d),
\end{equation}
where $I_x^k$ is the set of all intervals $I$ in $\mathbb{R}_{x_k}$ containing $x_k$.\

We will use extensively the following iterated maximal function:
\begin{equation}\label{Max1}
\cM_\theta f(x):=\left(M_d(\cdots(M_1|f|^\theta)\cdots)\right)^{1/\theta}(x),\;\theta>0,\;x\in\mathbb{R}^d.
\end{equation}

\begin{myre} If $R$ is a rectangle $R=I_1\times\dots\times I_n,$ it follows easily that for every locally integrable $f$
\begin{equation}\label{rectangle}
\int_R |f(y)| dy\leq |R| \cM_1 f(x)=|R|\cM_\theta^\theta|f|^{1/\theta}(x), \ \theta>0,\ x\in\bR^d.
\end{equation}
\end{myre}

Mixed-norm Lebesgue spaces will play an important role for defining mixed-norm smoothness spaces. Let $\vec{p}=(p_1,\dots,p_d)\in(0,\infty)^d$ and $f:\mathbb{R}^d\rightarrow \mathbb{C}$. We say that $f\in L_{\vec{p}}:=L_{\vec{p}}(\mathbb{R}^d)$ if
\begin{equation}\label{Lp}
\|f\|_{\vec{p}}:=\left(\int_\mathbb{R}\cdots\left(\int_\mathbb{R}\left(\int_\mathbb{R} |f(x_1,\dots,x_d)|^{p_1} dx_1\right)^{\frac{p_2}{p_1}} dx_2\right)^{\frac{p_3}{p_2}}\cdots dx_d\right)^{\frac{1}{p_d}}<\infty.
\end{equation}
The quasi-norm $\|\cdot\|_{\vec{p}},$ is a norm when $\min(p_1,\dots,p_d)\geq 1$ and turns $(L_{\vec{p}},\|\cdot\|_{\vec{p}})$ into a Banach space. Note that when $\vec{p}=(p,\dots,p),$ then $L_{\vec{p}}$ coincides with $L_p.$ For additional properties of $L_{\vec{p}}$, see for example \cite{Bagby:1975uc,MR126155}.

For $\vp=(p_1,\dots,p_d)\in (0,\infty)^d$, $0<q\leq \infty$, and  a sequence
$f=\{f_j\}_{j\in\bN}$ of $L_{\vec{p}}(\bR^d)$ functions, we define the (quasi-)norm
$$\|f\|_{L_{\vec{p}}(\ell_q)}:=\big\|\big(\sum_{j\in\bN} 
|f_j|^q\big)^{1/q}\big\|_{\vec{p}}.$$ 
Where there is no risk of ambiguity we will abuse notation slightly and write $\|f_k\|_{L_{\vec{p}}(\ell_q)}$ instead
of $\|\{f_k\}_k\|_{L_{\vec{p}}(\ell_q)}$.

 We shall need a variation of the Fefferman-Stein vector-valued maximal inequality, see \cite{MR2401510,Bagby:1975uc}. Suppose $\vec{p}=(p_1,\dots,p_d)\in(0,\infty)^d,\ 0<q\leq \infty$ and $0<\theta<\min(p_1,\dots,p_d,q)$, then
 
\begin{equation}
  \label{eq:fs}
  \|\{(\cM_\theta(f_j)\}\|_{L_{\vec{p}}(\ell_q)}\leq C_B\|\{f_j\}\|_{L_{\vec{p}}(\ell_q)},
\end{equation}
with $C_B:=C_B(\theta,\vec{p},q)$.

We mention that using a slight variation on standard techniques, see e.g.\ \cite[Proposition 3.3]{Borup2008}, the following Petree type estimate can be derived: For every $\theta>0,\;$there exists a constant $c=c_\theta>0,$ such that for every $t>0$, and $f$ with $\supp(\hat{f})\subset t^{\vec{a}}[-2,2]^d,$
\begin{equation}\label{M3}
\sup_{y\in\mathbb{R}^n} \dfrac{|f(y)|}{\brac{t^{\vec{a}}(x-y)}^{\nu/\theta}}\leq c\cM_\theta f(x),\;x\in\mathbb{R}^d.
\end{equation}   

By combining the Petree estimate with the Fefferman-Stein vector-valued maximal inequality, calling on \eqref{brac_esti} and \eqref{brac_int}, and again using a very slight variation on standard techniques, see e.g.\ \cite[Theorem 3.5]{Borup2008}, we can derive the multiplier result given in Proposition \ref{th:vecmult} below, which will be used in Section \ref{sec:5}. The details of the proof are left to the reader.

For $\Omega=\{\Omega_n\}$ a sequence of compact
subsets of $\bR^d$, we let
$$L_{\vec{p}}^\Omega(\ell_q):=\{\{f_n\}_{n\in\bN}\in
L_{\vec{p}}(\ell_q)\,|\,\supp(\hat{f}_n)\subseteq\Omega_n,\,\forall n\}.$$ 
For $s\in \bR_+$, and $f:\bR^d\rightarrow \mathbb{C}$, we let
$$\|f\|_{H_2^m}:=\biggl( \int |\mathcal{F}^{-1}f(x)|^2 \brac{x}^{2m}
dx\biggr)^{1/2}$$
denote the (anisotropic) Sobolev  norm, where the Fourier transform is defined by
$\mathcal{F}(f)(\xi):=(2\pi)^{-d/2}\int_{\bR^d} f(x)e^{-ix\cdot\xi}\,dx$, $f\in L_1(\bR^d)$.  We have,

\begin{myprop}\label{th:vecmult}
  Suppose $\vp\in (0,\infty)^d$ and
  $0<q\leq \infty$, and let $\Omega=\{T_k \mathcal{C}\}_{k\in\bN}$ be
  a sequence of compact subsets of $\bR^d$ generated by a family
  $\{T_k=t_k^{\vec{a}} \cdot+\xi_k\}_{k\in\bN}$ of invertible affine
  transformations on $\bR^d$, with $\mathcal{C}$ a fixed compact subset of
  $\bR^d$. Assume $\{\psi_j\}_{j\in \bN}$ is a
  sequence of functions satisfying $\psi_j\in H^m_2$ for some $m>\frac
  \nu2+\frac{\nu}{\min(p_1,\ldots,p_d,q)}$. Then there exists a constant $C<\infty$
  such that 
$$\|\{\psi_k(D)f_k\}\|_{L_{\vec{p}}(\ell_q)} \leq
C\sup_j\|\psi_j(T_j\cdot)\|_{H^m_2}\cdot
\|\{f_k\}\|_{L_{\vec{p}}(\ell_q)}$$
for all $\{f_k\}_{k\in \bN} \in L_{\vec{p}}^\Omega(\ell_q)$, where $\psi_k(D)f:=\mathcal{F}^{-1}\{\psi_k\mathcal{F}(f)\}$.
\end{myprop}

\subsection{Schwartz functions and tempered distributions}

Finally, let us recall some basic facts about Schwartz functions and tempered distributions. We denote by $\cS=\cS(\bR^d)$ the Schwartz space of rapidly decreasing, infinitely differentiable functions on $\bR^d$. A function $\varphi\in\cC^{\infty}$ belongs to $\cS,$ when for every $k\in\bN_0$ and every multi-index $\alpha\in\bN_0^d$, with $\bN_0:=\bN\cup\{0\}$, there exists $c=c_{k,\alpha}>0$ such that
\begin{equation}\label{Snorm1}
|\partial ^\alpha \varphi(x)|\le c(1+|x|)^{-k},\;\;\text{for every}\;x\in\bR^d.
\end{equation}
The dual space $\cS'=\cS'(\bR^d)$ of $\cS$ is the space of tempered distributions.
We will denote by
$$\cS_{\infty}:=\cS_{\infty}(\bR^d)=\Big\{\psi\in\cS:\int_{\bR^d} x^\alpha \psi(x)dx=0,\;\forall \alpha \in \bN_0^d\Big\}.$$
We notice that $\cS_{\infty}$ is a Fr\'{e}chet space, because it is closed in $\cS$ and its dual is $\cS_{\infty}'=\cS'/\cP,$ where $\cP$ the family of polynomials on $\bR^d$. The anisotropic homogeneous mixed-norm Besov and Triebel-Lizorkin spaces will be based on distributions in $\cS'/\cP.$

% We have the following estimates derived from (\ref{brac_int}):
% \begin{mylemma}\label{l:int new} Let $\tau>\nu,$ then there exists a constant $c=c_{\tau}>0$ such that
% \begin{equation}\label{inew}
% {\rm{(i)}}\int_{\mathbb{R}^n}\brac{2^{k\vec{a}}y}^{-\tau}dy\le c2^{-k\nu},
% \end{equation}
% \begin{equation}\label{inew2}
% {\rm{(ii)}}\int_{\mathbb{R}^n}\brac{2^{k\vec{a}}(x-u)}^{-\tau} \brac{2^{k\vec{a}}(y-u)}^{-\tau}du\le c2^{-k\nu}\brac{2^{k\vec{a}}(x-y)}^{-\tau},
% \end{equation}
% for every $k\in \mathbb{Z}$.
% \end{mylemma}

\section{Dyadic rectangles, the {$\phi$}-transform, and mixed-norm smoothness spaces}

For $j\in\mathbb{Z}$ and $k\in\mathbb{Z}^d$, we denote by $Q_{j k}$ the dyadic rectangle
\begin{equation*}
Q_{j k}=\{(x_1,\dots,x_n)\in\mathbb{R}^d:k_i\leq 2^{ja_i}x_i<k_i+1,\;i=1,\dots,d\}.
\end{equation*}
For every $Q=Q_{j k}$ we denote by $x_Q=2^{-j\vec{a}}k$ the ``lower left-corner", and by $|Q|=2^{-\nu j}$, the volume. We also notice that
$$Q_{j k}=2^{-j\va}([0,1)^d)+x_{Q_{j k}}$$
For now on we will denote by $\mathcal{Q}_{j}$ the set of all dyadic rectangles of volume $|Q|=2^{-j\nu},\;j\in\bZ$ and by $\cQ$ the set of all dyadic rectangles. Note that for every $j\in\bZ,$ the set of all the dyadic rectangles of the same volume $\mathcal{Q}_{j}$ is a disjoint partition of $\bR^d.$

We now briefly present the resolution of the identity that will be needed for the $\varphi$-transform and will form the framework for introducing the mixed-norm Besov and Tiebel-Lizorkin spaces. 
It is known, see e.g. \cite[Lemma 6.1.7]{Bergh1976}, that we can choose  $\varphi, \psi$ satisfying the following conditions
\begin{equation}\label{phi1}
\varphi,\psi\in\cS(\mathbb{R}^d),
\end{equation}
\begin{equation}\label{phi2}
\text{supp} \ \hat{\varphi},\hat{\psi}\subseteq\{\xi\in\mathbb{R}^d:\;2^{-1}\leq|\xi|\leq2\},
\end{equation}
\begin{equation}\label{phi3}
|\hat{\varphi}(\xi)|, |\hat{\psi}(\xi)|\geq c>0 \ \ \ \text{if} \ \ 2^{-3/4}\leq |\xi|\leq2^{3/4},
\end{equation}
and
\begin{equation}\label{phi4}
\sum\limits_{j\in \mathbb{Z}} \overline{\hat{\varphi}(2^{-j\vec{a}}\xi)}\hat{\psi}(2^{-j\vec{a}}\xi)=1 \ \ \ \text{if} \ \ \xi\neq 0.
\end{equation}

We then put $\varphi_{j}(x)=2^{j\nu }\varphi(2^{j\vec{a}}x)$ and define $\psi_{j},\;j\in\mathbb{Z}$, in a similar way.
We  have $\widehat{\varphi_{j}}(\xi)=\widehat{\varphi}(2^{-j\vec{a}}\xi)$ for every $\xi\in\mathbb{R}^d,$ so by (\ref{phi2})
\begin{equation}\label{phi6}
\supp(\widehat{\varphi_j}),\;\supp(\widehat{\psi_j})\subseteq2^{j\vec{a}}\{\xi\in\mathbb{R}^d:2^{-1}\le|\xi|\leq2\}=:T_j.
\end{equation}
We deduce from (\ref{ad1})-(\ref{ad8}) that there exists an $M\in\mathbb{N}$, such that $T_j\cap T_i= \emptyset$ when $|i-j|>M$.\

For every $\tau >0$ there exists a constant $c_\tau>0$ such that
\begin{equation}\label{phi40}
|\varphi_j(x)|, |\psi_j(x)|\leq c_\tau 2^{j\nu}\brac{2^{j\vec{a}}x}^{-\tau},\qquad x\in\bR^d.
\end{equation}

\begin{mydef}\label{adm}
Functions $\varphi,\psi$ satisfying {\rm(\ref{phi1})-(\ref{phi3})} will be called admissible.
\end{mydef}

\subsection{Mixed norm smoothness spaces}

We now recall the definition of  anisotropic mixed-norm Triebel-Lizorkin and Besov spaces.

For  $\varphi,\psi$ are admissible, we have the following definition of the homogeneous anisotropic mixed-norm Triebel-Lizorkin and Besov spaces

\begin{mydef}
Let $s\in\mathbb{R},\;\vec{p}\in(0,\infty)^d,\;q\in(0,\infty]$ and $\vec{a}\in(0,\infty)^d.$ The anisotropic homogeneous mixed-norm Besov space $\dot{B}^s_{\vec{p}q}(\vec{a})$ is defined as the set of all $f\in \mathcal{S}'(\bR^d)/ \mathcal{P}$ such that 
\begin{equation}\label{Bnorm}
\|f\|_{\dot{B}^s_{\vec{p}q}(\vec{a})}:= \Big( \sum _{j\in\mathbb{Z}} (2^{sj}\|\varphi_{j}\ast f\|_{\vec{p}})^q \Big) ^{1/q}<\infty,
\end{equation}
and the anisotropic homogeneous mixed-norm Triebel-Lizorkin space $\dot{F}^s_{\vec{p}q}(\vec{a})$ is defined as the set of all $f\in \mathcal{S}'(\bR^d)/ \mathcal{P}$ such that 
\begin{equation}\label{TLnorm}
\|f\|_{\dot{F}^s_{\vec{p}q}(\vec{a})}:=\Big \| \Big( \sum _{j\in\mathbb{Z}} (2^{sj}|\varphi_{j}\ast f|)^q \Big) ^{1/q}\Big \|_{\vec{p}}<\infty,
\end{equation}
with the $\ell_q$-norm replaced by $\sup _{j}$ if $q=\infty$ for both $\dot{B}^{s}_{\vec{p}q}(\vec{a})$ and $\dot{F}^{s}_{\vec{p}q}(\vec{a})$.
\end{mydef}
We will also need the associated discrete versions of the functions spaces.
\begin{mydef}\label{def:dT-L}
For $s\in\mathbb{R},\;\vec{p}=(p_1,\dots,p_d)\in(0,\infty)^d,\;q\in(0,\infty]$ and $\vec{a}\in(0,\infty)^d$ we define the sequence space $\dot{b}^s_{\vec{p}q}(\vec{a})$, as the set of all complex-valued sequences $a=\{a_Q\}_{Q\in\mathcal{Q}}$ such that
\begin{equation}\label{dBnorm}
\|a\|_{\dot{b}^s_{\vec{p}q}(\vec{a})}:=\Big(\sum_{j \in \mathbb{Z}} \Big\|\sum_{Q\in \cQ_{j}} |Q|^{-s/\nu}|a_R|\widetilde{\mathds{1}}_Q(\cdot)\Big\|_{\vec{p}}^q \Big)^{1/q} <\infty,
\end{equation}
and the sequence space $\dot{f}^s_{\vec{p}q}(\vec{a})$, as the set of all complex-valued sequences $a=\{a_Q\}_{Q\in\mathcal{Q}}$ such that
\begin{equation}\label{dTLnorm}
\|a\|_{\dot{f}^s_{\vec{p}q}(\vec{a})}:=\Big \|\Big(\sum_{j \in \mathbb{Z}} \sum_{Q\in \cQ_{j}} (|Q|^{-s/\nu}|a_R|\widetilde{\mathds{1}}_Q(\cdot))^q \Big)^{1/q}\Big \|_{\vec{p}} <\infty,
\end{equation}
where $\widetilde{\mathds{1}}_Q=|Q|^{-1/2} \mathds{1}_Q,$ the function $\mathds{1}_Q$ is the characteristic of the rectangle $R$ and the $\ell_q$-norm is replaced by the $\sup_{j}$ if $q=\infty$.
\end{mydef}

We are now ready to define the $\varphi$-transform $S_\varphi$ that provides an important and  universal discrete decomposition of the various smoothness spaces on the Besov and Triebel-Lizorkin sacale.   
\begin{mydef}\label{phi-transform}
Let $\varphi$ be an admissible function. The $\varphi$-transform $S_\varphi,$ or the analysis operator, is the map sending each $f\in\cS'/\mathcal{P}$  to the complex-valued sequence
\begin{equation*}
S_\varphi f=\{(S_\varphi f)_Q\}_Q, \ \text{with} \ (S_\varphi f)_Q= \langle f, \varphi_Q \rangle, \ \text{for every } Q\in \mathcal{Q}.
\end{equation*}

Let $\psi$ be admissible. The so called inverse $\varphi$-transform $T_\psi,$ or the synthesis operator, is the map taking a sequence $a=\{a_Q\}_Q$ to $T_\psi a=\sum \limits_{Q\in\mathcal{Q}} a_Q\psi_Q$.
\end{mydef}

\begin{myre} It is not entirely trivial  that $T_\psi$ is well defined. For an approach  to this question in the mixed-norm case, see \cite{MR3669781}.
\end{myre} 
One of the main results presented in \cite{MR3981281,Cleanthous:2017}, extending the seminal work by Frazier and Jawerth  \cite{Frazier1985,Frazier1990},  is the following discrete decomposition and norm characterisation using  the $\varphi$-transform.

\begin{mytheorem}\label{th:main}
Let $s\in\mathbb{R},\;\vec{p}=(p_1,\dots,p_d)\in(0,\infty)^d,\;q\in(0,\infty],\;\vec{a}\in[1,\infty)^n$ and $\varphi, \psi$ admissible. The $\varphi$-transform $S_\varphi:\dot{F}^{s}_{\vec{p}q}(\vec{a})\rightarrow \dot{f}^{s}_{\vec{p}q}(\vec{a})$ and the inverse $\varphi$-transform $T_\psi: \dot{f}^{s}_{\vec{p}q}(\vec{a})\rightarrow\dot{F}^{s}_{\vec{p}q}(\vec{a})$ are bounded. Furthermore $T_\psi\circ S_\varphi$ is the identity on $\dot{F}^{s}_{\vec{p}q}(\vec{a})$. In particular, $\|f\|_{\dot{F}^{s}_{\vec{p}q}(\vec{a})}\sim \|S_\varphi f\|_{\dot{f}^{s}_{\vec{p}q}(\vec{a})}$ for every $f\in\dot{F}^{s}_{\vec{p}q}(\vec{a})$. The same characterisation result holds for the Besov spaces $\dot{B}^{s}_{\vec{p}q}(\vec{a})$ with associated sequence space  $\dot{b}^{s}_{\vec{p}q}(\vec{a})$.
\end{mytheorem}

One potential issue with the decomposition system $\{\psi_R\}_R$ is the redundancy of  the system, which makes the analysis of the systems approximation properties much more difficult, see the discussion in \cite{MR1794807}. The goal in the sequel is to construct an orthonormal basis for $L_2(\bR^d)$, i.e., a non-redundant system, that will support a similar type of universal decomposition of $\dot{B}^{s}_{\vec{p}q}(\vec{a})$  and  $\dot{F}^{s}_{\vec{p}q}(\vec{a})$, respectively, as provided by Theorem \ref{th:main}.
\section{Orthonormal   brushlet bases} \label{sec:onb}

A universal method to create orthonormal bases for multivariate $L_2$ is to use 
a tensor product construction based on a univariate orthornomal basis for $L_2(\bR)$. Often this method works well, but the  approach can be considered problematic for wavelet bases and similar systems as basis elements with long "skinny" support in the frequency plane are created. Such elements are often not  well-adapted for the analysis of smoothness spaces of Besov or Triebel-Lizorkin type due to the incompatible time-frequency structure. In this section, we will still rely on a tensor product construction, but it will be based on a more stable decomposition of the frequency space. The specific coverings will be introduced in Section \ref{sec:cov}, and the corresponding bases will be based on tensor products of so-called univariate brushlets. The construction of brushlets is a well-establish approach \cite{Laeng1990,Meyer1997a}, but for the convenience of the reader we have included a summary of the needed facts about brushlets in Appendix \ref{app:brush}.

\subsection{Lizorkin coverings of {$\bR^d$} and associated decomposition systems}\label{sec:cov}
In order to combine individual multivariate brushlet systems in a coherent way to obtain an orthonormal basis for the full space $L^2(\bR^d)$,  it will be convenient to have a suitable partition made up of rectangles  of the frequency space $\bR^d\backslash \{0\}$. Moreover, the partition should also be constructed to be compatible with the frequency structure of the anisotropic Triebel-Lizorkin space. We will now introduce such partitions based on an idea  first considered by Lizorkin in an isotropic setting.

 Given an anisotropy
$\mathbf{a}=(a_1, \ldots, a_d)$, define the rectangles
$$R_j = \{x\colon |x_i|\leq 2^{ja_i},\; i=1,\ldots, d\},\qquad
\text{for }\; j\in \bZ,$$
and the corridors $K_j = R_j\setminus R_{j-1}$ for $j\in \bZ$.
Let $E_2=\{\pm 1,\pm 2\}$, $E_1:=\{\pm 1\}$, and put $E:=
  E_2^n\setminus E_1^n$, where we notice that $|E|=4^n-2^n$.  For each $j\in \bZ$,
and $k\in E$, define
$$R_{j,k}=\{ x\in \bR^d\colon \sgn(x_i)=\sgn(k_i),\; \text{and}\;
(|k_i|-1)2^{j-1}\leq |x_i|^{1/a_i}< |k_i|2^{j-1}\},$$ where the construction ensures that $K_j = \cup_{k\in E}R_{j,k}$.
Clearly, $R_{j,k}$ is a $d$-dimensional rectangle, so we may write
$$R_{j,k}=I^1_{j,k_1}\times I^2_{j,k_2}\times\cdots\times I_{j,k_d}^d,\qquad j\in\bZ,k\in E,$$
where each $I^i_{j,k}$ is an half-open interval in $\bR$. Letting $c_{j,k}\in\bR^d$ denote the center of $R_{j,k}$, we also have the affine representation
$$R_{j,k}=2^{j\va}B(k)\big([-1,1)^d\big)+c_{j,k},      $$
where $B(k):\bR^d\rightarrow \bR^d$, $k\in E$, is the matrix given by
$$B(k)=\text{diag}(b(k)_1,\ldots,b(k)_d),\quad\text{with }\quad b(k)_i:=\begin{cases}2^{-(a_i+1)},&\text{if }|k_i|=1\\
(1-2^{-a_i})/2,&\text{if }|k_i|=2.\end{cases}$$
The family $\cR:=\{R_{j,k}\}_{j\in\bZ,k\in E}$ gives an anisotropic Lizorkin partition
of $\bR^d\backslash\{0\}$. For notational convenience, we also define $\cR_j:=\{R_{j,k}\}_{k\in E}$.

We now construct univariate brushlet bases, following the outline in Appendix \ref{app:brush}, adapted to the intervals  $\{I_{j,k_i}^i\}_{k\in E}$. More specifically, we consider the intervals,
\begin{align}\label{eq:part}
 \big\{\pm \big[0,2^{(j-1)a_i}\big)\big\}\cup
\big\{ \pm \big[2^{(j-1)a_i},2^{ja_i}\big)\big\}
\end{align}
where we choose corresponding cutoff radii  $2^{(j-2)a_i}$ at $\pm 2^{ja_i}$ and  cutoff radii $2^{(j-3)a_i}$ at $\{0,\pm 2^{(j-1)a_i}\}$, where we again refer to Appendix \ref{app:brush} for a discussion of cutoff radii.

For each $R_{j,k}=I^1_{j,k_1}\times I^2_{j,k_2}\times\cdots\times I_{j,k_d}^n$, $j\in\bZ,k\in E$, we define the orthonormal brushlet system, 
$$\{w_{n,R_{j,k}}:n\in \bN_0^d\},$$
with $w_{n,R_{j,k}}$ defined as the  tensor product in Eq.\ \eqref{eq:multivariate}, see Appendix \ref{app:brush},  $$w_{n,R_{j,k}}:=\bigotimes_{i=1}^d w_{n_i,I^i_{j,k_i}}.$$
\begin{myre}\label{rem:wd}
It is important to observe that by the construction in Appendix \ref{app:brush}, $w_{n,R_{j,k}}\in \cS_\infty(\bR^d)$, since the (compact) frequency support of the smooth function $w_{n,R_{j,k}}$ does not contain the zero frequency. Hence, the expansion coefficient $\langle f, w_{n,R_{j,k}}\rangle$ is well-defined for $f\in \cS'/\cP$ (and, in particular, for $f\in \dot{F}_{\vp q}^s\cup\dot{B}_{\vp q}^s$). 
\end{myre}
We claim that the full system
\begin{equation}\label{eq:W}
    \cW:=\{w_{n,R_{j,k}}:n\in \bN_0^d,j\in\bZ,k\in E\},
\end{equation}
is an orthonormal basis for $L_2(\bR^d)$ that provides the wanted universally stable decomposition of the Besov and Triebel-Lizorkin spaces. The claim will be substantiated in the following sections.
\section{Stability of the multivariate brushlet systems}\label{sec:5}
We have the following fundamental property of the system $\{w_{n,R_{j,k}}\}$ defined in the previous section.

\begin{myprop}\label{prop:onb}
Given any anisotropy
$\mathbf{a}=(a_1, \ldots, a_d)\in (0,\infty)^d$, the corresponding system $\cW$ defined in Eq.\ \eqref{eq:W} forms an orthonormal basis for $L_2(\bR^d)$.
\end{myprop}
The proof on the Proposition can be found in Appendix \ref{s:app}.
% Let us conclude this section by introducing some additional notation. Put $$\mathcal{R}^\alpha:=\bigcup_{j\in\bZ} \mathcal{A}_j.$$
% For an arbitrary rectangle $R=I\times J\in \mathcal{R}^\alpha$, we let $\xi_R\in R$ denote the mid-point of $R$. Put $\mathcal{R}_0:=[-1/2,1/2]^2$, and note that 
% \begin{equation}\label{eq:affine}R=\delta_R(\mathcal{R}_0)+\xi_R,\end{equation}
% where $\delta_R:=\text{diag}(|I|,|J|)$. This shows that $\mathcal{R}^\alpha$ is a so-called structured covering in the terminology of \cite{AlJawahri2019}. Notice also that the covering satisfies the geometric rule
% \begin{equation}\label{eq:geom}
%     |R|\asymp |\xi_R|^{\beta(R)},\qquad \beta(R)=\begin{cases}2\alpha,&R\in \cup_{j>0} \mathcal{A}_j\\
% 2(2-\alpha),&R\in \cup_{j<0} \mathcal{A}_j\end{cases}.
% \end{equation}

% Let us recall that class $\mathcal{S}_0:=\mathcal{S}_0(\mathbb{R}^2)$, which is the closed subspace of the Schwartz class $\mathcal{S}(\mathbb{R}^2)$, is defined by
% \begin{align*}
% \mathcal{S}_0&=
% \left\{f\in\mathcal{S}(\mathbb{R}^2): \int f(x)\cdot x^\alpha\,dx=0\, \text{ for all } \alpha\in\N_0^2\right\}.
% \end{align*}

% It can be proved
% that $\dot{F}_{p,q}^{s,\alpha}(\bR^2)$ $[\dot{M}_{p,q}^{s,\alpha}(\bR^2)]$ satisfy
% $$\mathcal{S}_0\hookrightarrow \dot{M}_{p,q}^{s,\alpha}(\bR^2)\hookrightarrow
% \cS'\backslash\mathcal{P},\qquad \cS_0\hookrightarrow \dot{F}_{p,q}^{s,\alpha}(\bR^2)\hookrightarrow
% \cS'\backslash\mathcal{P},$$ see \cite{AlJawahri2019}. Moreover, if $p,q<\infty$, $\mathcal{S}_0$ is dense in
% $M^{s,\alpha}_{p,q}(\bR^2)$.

\subsection{Characterizations of $\dot{M}_{\vp,q}^{s}(\va)$ and $\dot{F}_{\vp,q}^{s}(\va)$}
We claim that the spaces $\dot{F}_{\vp,q}^{s}(\va)$ and $\dot{M}_{\vp,q}^{s}(\va)$ can be completely characterized using the brushlet system considered in Proposition \ref{prop:onb}. We focus on proving this claim for the Triebel-Lizorkin type spaces $\dot{F}_{\vp,q}^{s}(\va)$. For many of the proofs in this Section, we will call on results on vector-valued multiplies as discussed in Section \ref{sec:multi}. The corresponding results for the Besov spaces $\dot{M}_{\vp,q}^{s}(\va)$ are easier to handle due to the (simpler) structure of the space, and we will just state the results for $\dot{M}_{\vp,q}^{s}(\va)$ and leave most details to the reader.

For $R:=R_{j,k}=I^1_{j,k_1}\times I^2_{j,k_2}\times\cdots\times I_{j,k_d}^d\in\cR$, we define for $\cn\in \bN_0^d$,

\begin{equation}
  \label{eq:QD}
  U(R,\cn)=\left\{y\in\bR^d:
2^{j\vec{a}} y-\pi{\left(\cn+\ca\right)}\in [-1,1]^d\right\},
\end{equation}
where $\ca:=[\frac12,\ldots,\frac12]^T\in\bR^d$.
It is easy to verify there exists $L<\infty$ so that uniformly in $x$
and $R$,  $\sum_n\mathds{1}_{U(R,\cn)}(x)\leq L$. One may also verify that for $\cn,\cn'\in \bN_0^d$, $U(R,\cn')=U(R,\cn)+\pi2^{-j\vec{a}}(\cn'-\cn)$, and that $|U(R,n)|=2^{-j\nu}\asymp |R|^{-1}$.

%We also note that the 
Let us now prove that the canonical coefficient operator is bounded on
$\dot{F}^s_{\vp,q}(\va)$. We mention that the operator is well-defined according to  Remark \ref{rem:wd} in the sense that the quantity $\mathcal{S}_q^s(f)$ defined below in Proposition \ref{prop:normchar} can be computed for any $f\in\cS'/\mathcal{P}$.

\begin{myprop}\label{prop:normchar}
% Let $\qp=\{T_k \mathcal{C}\}_{k\in\bN}$ be
%  a structured admissible covering generated by a family

% , with $\mathcal{C}$ a fixed compact subset of $\bR^d$. Let $w$ be a
% $\qp$-moderate weight.
Suppose $s\in\bR$, $\vp\in (0,\infty)^d$, and $0<q\leq \infty$. Then
$$\|\mathcal{S}_q^s(f)\|_{\vp}\leq C\|f\|_{\dot{F}^{s}_{\vp,q}(\va)},\qquad f\in \dot{F}^{s}_{\vp,q}(\va),$$
where
\begin{equation}
    \label{eq:Sq}
\mathcal{S}_q^s(f):=\Big(\sum_{R \in \mathcal{R}} \sum_{n\in \bN_0^d} (|R|^{s/\nu}|\langle
f,w_{\cn,R}\rangle|\widetilde{\mathds{1}}_{U(R,n)}(\cdot))^q \Big)^{1/q},
    \end{equation}
with $\widetilde{\mathds{1}}_{U(R,n)}:=|R|^{1/2}\mathds{1}_{U(R,n)}$. 
\end{myprop}

\begin{proof}
Take $f\in  \dot{F}^{s}_{p,q}(\R^d)\cap \cS_\infty(\bR^d)$ and consider   $R\in \mathcal{R}_{j}$,  for some $j\in \bZ$, i.e.,  $|R|\asymp 2^{j\nu}$ (uniformly in $j$). We write the cosine term in Eq.\ \eqref{eq:brush1} as a sum of complex exponentials, and we take a tensor product to create $w_{\cn,R}$ . This process creates a multivariate function with $2^d$ "humps", and, as it turns out, we will consequently need $2^d$ terms to control the inner product $\langle
f,w_{\cn,R}\rangle$. By \eqref{eq:G},
\begin{align*}|\langle
f,w_{\cn,R}\rangle|&\leq C 2^{-d}|R|^{1/2}\sum_{m=1}^{2^d} |\langle f,G_R(\Delta(x+U_m e_{n,R}))\rangle|\\
&=C 2^{-d}|R|^{1/2}\sum_{m=1}^{2^d} |f*{H_R}(U_m e_{n,R})\rangle |,
\end{align*}
where ${H_R}(\cdot):=G_R(-\Delta\, \cdot)$. Now, notice for $x\in U(R,n)$ we have 
 $\brac{2^{j\vec{a}}(e_{n,R}-x)}\leq c$. Hence, for $0<t<\min(p_1,\dots,p_d,q),$
\begin{align*}
\sum_{m\in \bN_0^d} (|R|^{s/\nu}|\langle
f,w_{\cn,R}\rangle|&\widetilde{\mathds{1}}_{U(R,m)}(\cdot))^q\\
&\leq \sum_{m=1}^{2^d}\sum_{n\in \bN_0^d}|R|^{sq/\nu} |f*{H_R}(U_m e_{n,R})\rangle |^q\mathds{1}_{U(R,n)}(x)
\nonumber
\\
&\le c\sum_{m=1}^{2^d}\sum_{n\in \bN_0^d}\frac{2^{sqj}|f*H_R(U_me_{n,R})|^q}{\brac{2^{j\vec{a}}(e_{n,R}-x)}^{\nu q/t}}\mathds{1}_{U(R,n)}(x)
\nonumber
\\
&\le c\sum_{m=1}^{2^d}\sum_{n\in \bN_0^d}\Big(\sup_{y\in U(R,n)}\frac{2^{j s}|f*H_R(U_my)|}{\brac{2^{j\vec{a}}(y-x)}^{\nu/t}}\Big)^q\mathds{1}_{U(R,n)}(x).
\label{Thmainp6}
\end{align*}
We now apply  the maximal inequality (\ref{M3}), where we notice by the construction of $H_R$, $\widehat{H_R}\subseteq 2^{j\vec{a}}[-2,2]^d$, so consequently  $\text{supp} [((f*H_R)\circ U_m))^\wedge]\subseteq 2^{j\vec{a}}[-2,2]^d$. We obtain
\begin{equation}\label{Thmainp7}
\sum_{n\in \bN_0^d} (|R|^{s/\nu}|\langle
f,w_{\cn,R}\rangle|\widetilde{\mathds{1}}_{U(R,n)}(\cdot))^q \le c
\sum_{m=1}^{2^d} \big(2^{j s}\cM_\theta\big(f*H_R)\big)(U_mx)\big)^q,
\end{equation}
for every $x\in\mathbb{R}^d.$ 
We now pass to the discrete Triebel-Lizorkin norm of $S_{\varphi}f$, 
where we notice that, due to the frequency localization of $H_R$ around $R$,
\begin{equation}
    \label{eq:HR}
f*H_R=\sum_{i\in F_R} \phi_i*\psi_i*f*H_R,
\end{equation}
for $F_R:=\{i: T_i\cap \text{supp}(\widehat{H_R})\not=\emptyset\}\subset j+\{-M,-M+1,\ldots,M-1,M\}$, where the dyadic structure of $\cup_j\mathcal{P}_j$, and the support sets $\{T_j\}$, ensure that $M$ can be chosen independent of $R$. 

One can also verify, using the specific structured compact support in frequency, that
both systems $\{\psi_j\}_i$ and $\{H_R\}_{R\in\cR}$ satisfy the conditions stated in  Proposition \ref{th:vecmult} for any $m>\frac
  \nu2+\frac{\nu}{\min(p_1,\ldots,p_d,q)}$. 
Hence, by substituting \eqref{eq:HR} in (\ref{Thmainp7}), and using Proposition \ref{th:vecmult} twice, we obtain

\begin{eqnarray*}\|\mathcal{S}_q^s(f)\|_{\vp}
&\le&
c\sum_{m=1}^{2^d}\Big\|\Big(\sum_{R \in \mathcal{R}}\big(2^{js}\big(\sum_{i\in F_R}\phi_i*\psi_i*f*H_R\big)(U_m\cdot)\big)^q \Big)^{1/q}\Big \|_{\vec{p}}\nonumber\\
&\le&
c\sum_{m=1}^{2^d}\Big\|\Big(\sum_{R \in \mathcal{R}}\sum_{i\in F_R}\big(2^{js}\big(\phi_i*\psi_i*f*H_R\big)(\cdot)\big)^q \Big)^{1/q}\Big \|_{\vec{p}}\nonumber\\
&\le&
c\sum_{m=1}^{2^d}\Big\|\Big(\sum_{R \in \mathcal{R}}\sum_{i\in F_R}\big(2^{js}\big(\phi_i*\psi_i*f\big)(\cdot)\big)^q \Big)^{1/q}\Big \|_{\vec{p}}\nonumber\\
&\le&c\Big\|\Big(\sum_{j\in \mathbb{Z}}(2^{sj}|{\varphi}_{j}*f|)^q \Big)^{1/q}\|_{\vec{p}}\\&\le& c\|f\|_{\dot{F}^{s}_{\vec{p}q}(\vec{a})},
\nonumber
\end{eqnarray*}
and the general result then follows from a limiting argument using that $\cS_\infty(\bR^d)$ is dense in $\dot{F}^{s}_{\vec{p}q}(\vec{a})$.\ 
\end{proof}

\begin{myre}
Clearly, the quantity $\|\mathcal{S}_q^s(f)\|_{\vp}$ is closely related to 
the $\dot{f}_{\vp q}^s$-norm given in \eqref{dTLnorm}, but slightly different domains for the sequences are used in the two cases due the structure imposed by the Lizorkin-decomposition. Rather than applying a potentially complicated re-indexing, we will slightly abuse notation below and (re-)define 
\begin{equation}\label{eq:redef}
\|\{c_{n,R}\}_{n,R}\|_{\dot{f}_{\vp q}^s(\va)}:=\bigg\|\Big(\sum_{R \in \mathcal{R}} \sum_{n\in \bN_0^d} (|R|^{s/\nu}|c_{n,R}|\widetilde{\mathds{1}}_{U(R,n)}(\cdot))^q \Big)^{1/q}\bigg\|_{\vp},
    \end{equation}
and
\begin{equation}\label{eq:redBnorm}
\|\{c_{n,R}\}_{n,R}\|_{\dot{b}^s_{\vec{p}q}(\vec{a})}:=\Big(\sum_{R \in \cR} \Big\|\sum_{n\in \bN_0^d} |R|^{s/\nu}|c_{n,R}||\widetilde{\mathds{1}}_{U(R,n)}(\cdot))\Big\|_{\vec{p}}^q \Big)^{1/q},
\end{equation}
for sequences $\{c_{n,R}\}_{n,R}$ with index domain $\bN^d_0\times\cR$.
\end{myre}

With the adjusted notation, Proposition \ref{prop:normchar} can be rephrased more naturally proving that the canonical analysis operator $S_{\cW}$, defined for $f\in\cS'/\mathcal{P}$  by
\begin{equation*}
S_{\cW} f=\{(S_{\cW} f)_{n,R}\}_{n,R}, \ \text{with} \ (S_{\cW} f)_{n,R}= \langle f, w_{n,R} \rangle, \ \text{for } (n,R)\in \bN_0^d\times\cR,
\end{equation*}
is bounded on  $\dot{F}_{\vp q}^s(\va)$,
\begin{equation}\label{eq:SW}
S_{\cW}:\dot{F}_{\vp q}^s(\va)\longrightarrow \dot{f}_{\vp q}^s(\va),
\end{equation}
with a similar result for $\dot{B}_{\vp q}^s(\va)$ and $\dot{b}_{\vp q}^s(\va)$.

We now turn to a companion result to Proposition \ref{prop:normchar} and \eqref{eq:SW} that will imply that the canonical reconstruction operator $T_{\cW}$ for $\{w_{n,R}\}$, defined for (finite) sequences $c=\{c_{n,R}\}$ by $$T_{\cW}c:=\sum_{n,R} c_{n,R} w_{n,R},$$ is bounded on $\dot{f}_{\vp q}^s(\va)$ and $\dot{b}_{\vp q}^s(\va)$, respectively, using the interpretation specified in  \eqref{eq:redef}.
\begin{myprop}\label{lem:reconstruct}
  Suppose  $s\in \bR$, $\vp\in (0,\infty)^d$ and let $0<q< \infty$. Then for
  any finite sequence
  $s=\{s_{\cn,R}\}_{n,R}$, we have 
  $$\big\|T_{\cW}s\big\|_{\dot{F}_{p,q}^{s}(\va)} \leq C
  \|\{s\}\|_{\dot{f}^{s}_{p,q}(\va)}.$$
% \Bigl\| w(\xi_k)^s |T_k|^{1/2} \sum_n
%  |s_{k,n}| \chi_{R(k,n)}\Bigr\|_{L_p(\ell_q)}.
 A similar estimate holds for $\dot{B}_{p,q}^{s}(\va)$ and $\dot{b}_{p,q}^{s}(\va)$.
\end{myprop}

\begin{proof}
Assume that $\phi,\psi$ is an admissible pair as specified in Definition \ref{adm}. We will use the associated system  $\{\phi_j\}_{j\in \bZ}$ to calculate the norm in $\dot{F}_{p,q}^{s}(\va)$. Using the structure given by \eqref{eq:brush1}, and Proposition \ref{th:vecmult}, we get
  \begin{align*}
    \Bigl\|\sum_{n,R} s_{\cn,R}w_{n,R}\Bigr\|_{\dot{F}_{\vec{p},q}^{s}} &= \Bigl\|
    \Big\{2^{js} \phi_{j}*\Big(\sum_{n,R}
    s_{\cn,R}w_{n,R}\Big)\Big\}_{j}\Bigr\|_{L_{\vec{p}}(\ell_q)}\\
    &\leq C\Bigl\|\Big\{
    2^{js} \sum_{R \in N_j} \sum_{\cn}
    s_{\cn,R}w_{n,R}\Big\}_j\Bigr\|_{L_{\vec{p}}(\ell_q)},
\end{align*}
where $N_j= \{R'\in \mathcal{P}\colon T_j\cap \supp
(b_{R'})\neq \emptyset\}$. It follows easily  that $\# N_j$ is
uniformly bounded, and
we obtain 
\begin{multline}\Bigl\|
   \Big\{ 2^{js} \sum_{R \in N_j} \sum_{\cn}
    s_{n,R}w_{n,R}\Big\}_j
\Bigr\|_{L_{\vec{p}}(\ell_q)}\\ \leq
    C\biggl\|\biggl( \sum_{R\in N_j} \Bigl( 2^{js} \sum_n
    |s_{n,R}| |w_{n,R}|\Bigr)^q\biggr)^{1/q}\biggr\|_{\vec{p}}.
    \end{multline}
Fix $0<r<\min(p_1,\ldots,p_d,q)$. Then by Lemma \ref{lem:maxbound} below, and the
Fefferman-Stein maximal inequality \eqref{eq:fs}, we obtain
\begin{align*}
    \Bigl\|\Big\{
    2^{js} & \sum_{\cn} |s_{\cn,R}| |w_{n,R}|\Big\}_j
\Bigr\|_{L_{\vec{p}}(\ell_q)}\\
&\leq  \Bigl\|\Big\{
    2^{js} \sum_{m=1}^{2^d} \sum_{\cn} |s_{\cn,R}| |G_R(\Delta(x-U_m e_{n,R})|\Big\}_R
\Bigr\|_{L_{\vec{p}}(\ell_q)}\\
    &\leq C\Bigl\|\Big\{
    2^{js} |R|^{1/2} \sum_{m=1}^{2^d} \mathcal{M}_r\Bigl(\sum_{\cn}
  |s_{\cn,R}|
   \mathds{1}_{U(R,\cn)}\Bigr)(U_{m}\cdot)\Big\}_R
\Bigr\|_{L_{\vec{p}}(\ell_q)}\\
&\leq C'\Bigl\|\Big\{
    2^{js} |R|^{1/2} \sum_{\cn}
  |s_{\cn,R}| \mathds{1}_{U(R,n)}\Big\}_R\Bigr\|_{L_{\vec{p}}(\ell_q)},
  \end{align*}
  where we used the (quasi-)triangle inequality and straightforward substitutions in the integrals.
The result now follows since the sum over $\cn$ is locally finite with a
uniform bound on the number of non-zero terms, which implies that
$$\Big(\sum_{\cn}
  |s_{n,R}| \mathds{1}_{U(R,\cn)}\Big)^q\asymp \sum_{\cn}
  |s_{\cn,R}|^q \mathds{1}_{U(R,n)},$$
uniformly in $R$, where we also recall that $\widetilde{\mathds{1}}_{U(R,n)}:=|R|^{1/2}\mathds{1}_{U(R,n)}$.
\end{proof}
The following estimate was used in the proof of Proposition \ref{lem:reconstruct}.
\begin{mylemma}\label{lem:maxbound}
  Let $0<r\leq 1$. There exists a constant $C$, independent of $R\in\cR$,  such that for any finite sequence $\{s_{n,R}\}_{n,R}$ we have
  $$\sum_{\cn} |s_{n,R}||w_{n,R}(x)|\leq C|R|^{1/2} \sum_{m=1}^{2^d} \cM_r\Bigl(\sum_{\cn}
  |s_{n,R}| \mathds{1}_{U(R,\cn)}\Bigr)(U_m x),\quad x\in\bR^d,$$
  with $U_m$ defined in Eq.\ \eqref{eq:um}.
\end{mylemma}
The proof of Lemma \ref{lem:maxbound} can be found in Appendix \ref{s:app}.
We now use Proposition \ref{prop:normchar} and Proposition
\ref{lem:reconstruct} to obtain the main result of this paper, that $\{w_{\cn,R}\}$ is an orthonormal basis that universally 
captures the norm of $\dot{B}^{s}_{\vp, q}(\va)$ and $\dot{F}^{s}_{\vp, q}(\va)$, respectively. Moreover, the system forms an unconditional basis for  $\dot{B}^{s}_{\vp, q}(\va)$ and for $\dot{F}^{s}_{\vp, q}(\va)$ in the Banach space case. 
% We define 
% the coefficient operator
% $C\colon F^s_{p,q}(h,w)\rightarrow f^s_{p,q}$ by
% $Cf=\{\langle f,\eta_{T_k,n}\rangle\}_{n,T}$ and 
% the reconstruction operator
% $R:f^s_{p,q}\rightarrow F^s_{p,q}(h,w)$, by
% $\{s_{k,n}\}_{k,n}\rightarrow \sum_{k,n} s_{k,n} |T_k|^{1/p-1/2}\eta_{T_k,n}$.

\begin{mytheorem}\label{prop:modu}
Let $s\in \bR$, $\vp\in (0,\infty)^d$, $0<q< \infty$. The maps $S_{\cW}$ and $T_{\cW}$ provide isomorphisms between 
$\dot{F}^{s}_{\vp,q}(\va)$ and $\dot{f}^{s}_{\vp,q}(\va)$, with $T_{\cW}\circ S_{\cW}=Id_{\dot{F}^{s}_{\vp,q}(\va)}$. In particular, we have  the norm characterization 
$$ \|f\|_{\dot{F}^{s}_{\vp,q}(\va)}\asymp  \|S_{\cW}\|_{\dot{f}^{s}_{\vp,q}(\va)},\qquad f\in \dot{F}^{s}_{\vp,q}(\va),$$
with $\dot{f}^{s}_{\vp,q}(\va)$ defined in \eqref{eq:redBnorm}.
Moreover, $\{w_{\cn,R}\}_{n\in\bN_0^d,R\in\cR}$ forms an unconditional basis for $\dot{F}^{s}_{\vp,q}(\va)$.

Similar characterisation
results holds for the Besov spaces $\dot{B}^{s}_{\vp,q}(\va)$ with associated sequence space $\dot{b}^{s}_{\vp,q}(\va)$.
\end{mytheorem}

\begin{proof}
  The norm characterization follows at once by combining Proposition \ref{prop:normchar} and Lemma \ref{lem:reconstruct}. The claim that the system forms an unconditional
basis when $1\leq p, q < \infty$ follows easily from the fact that  $\dot{F}^{s}_{\vp,q}(\va)$ is a Banach space, and
that finite expansions in  $\{w_{\cn,R}\}$ have uniquely determined coefficients giving us a norm characterization of such expansions by the $L_{\vp}$-norm of  $\mathcal{S}_q^{s}(\cdot)$.

The proof in the  Besov  space case $\dot{B}^{s}_{\vp,q}(\va)$ is similar, but somewhat more simple. We leave the details to the reader.
\end{proof}

We  mention that for a lattice preserving dilation, one can construct $r$-regular wavelets forming orthonormal bases for $L_2(\bR^d)$, see \cite{bownik_construction_2001,calogero_wavelets_1999}.  Theorem \ref{prop:modu} can easily be extended to this case with a simplified proof.

The norm characterisation obtained in Theorem \ref{prop:modu} may appear similar to the characterisation obtained for tight frames in \cite{Cleanthous:2017,Cleanthous2019}, but one should notice the important additional  fact that $\{w_{\cn,R}\}$ forms an unconditional basis. This fact has significant implications for, e.g., $m$-term nonlinear approximation using the system $\{w_{\cn,R}\}$, where the linear independence will allow one to prove inverse estimates of Bernstein type. Inverse estimates are currently out of reach for the redundant $\phi$-transform frames considered in \cite{Cleanthous:2017,Cleanthous2019}, see a discussion of this longstanding open problem in \cite{MR1794807}.
\section{Nonlinear {$m$}-term approximation}\label{sec:nl}
We now consider $m$-term nonlinear approximation with the system $\cW$ in $\dot{F}^{\beta}_{\vp,q}(\va)$. Let us first introduce some needed notation. 

We use the notation $V
\hookrightarrow W$ to indicate $V$ is continuously embedded in $W$ for two (quasi)normed spaces $V$ and
$W$, i.e.,  $V \subset W$ and there is a constant $C < \infty$ such
that $\|\cdot\|_W \leq C \|\cdot\|_V$.
 
 A dictionary $\cD=\{g_n\}_{n\in\bN}$ in $\dot{F}^{\beta}_{\vp,q}(\va)$, $\vec{p}\in (0,\infty)^d$, $0<q\leq \infty$, $\beta\geq 0$, is a countable collection of
quasi-normalized elements from $\dot{F}^{\beta}_{\vp,q}(\va)$ in the sense that  there exists $C\geq 1$ such that $$C^{-1}\leq \|g_n\|_{\dot{F}^{\beta}_{\vp,q}(\va)}\leq C,\qquad n\in\bN.$$ Given $\cD$, we consider the collection of
all possible $m$-term expansions with elements from $\cD$:
$$\Sigma_m(\cD):=\Big\{\sum_{i\in \Lambda} c_i g_i\: \Big|\: c_i\in\mathbb{C},
\#\Lambda\leq m\Big\}.$$ The error of the best $m$-term
approximation to an element $f\in\dot{F}^{s}_{\vp,q}(\va)$ is then
$$\sigma_m(f,\cD)_{\dot{F}^{s}_{\vp,q}(\va)}:=\inf_{f_m\in \Sigma_m(\cD)} \|f-f_m\|_{\dot{F}^{s}_{\vp,q}(\va)}.$$

%\subsection{Characterization of the approximation space}

%\begin{myre}
%It is known that $\dot{F}^{0}_{\vp,q}(\va)$

%\end{myre}

The approximation space $\mathcal{A}^\alpha_q(\dot{F}^{s}_{\vp,q}(\va),\cD)$ essentially
consists of all functions that can be approximated at the rate
$O(m^{-\alpha})$. The parameter $q$ is for fine-tuning only.

\begin{mydef}
%   For $\cD$ a dictionary in $\Lp$, the error of the best $n$-term
%   approximation of $f\in \Lp$ from $\Sigma_n(\cD)$ is
%   \[\sigma_n(f,\cD)_p=\inf_{S\in\Sigma_n(\cD)} \|f-S\|_p,\]
%   and
  The approximation space $\mathcal{A}^\alpha_q(\dot{F}^{\beta}_{\vp,q}(\va),\cD)$ is defined
  by
  \[|f|_{\mathcal{A}^\alpha_q(\dot{F}^{\beta}_{\vp,q}(\va),\cD)}:=\bigg(\sum_{m=1}^\infty
  \big(m^\alpha \sigma_m(f,\cD)_{\dot{F}^{\beta}_{\vp,q}(\va)}\big)^q
  \frac{1}{m}\bigg)^{1/q}<\infty,\]
  and (quasi)normed by
  $\|f\|_{\mathcal{A}^\alpha_q(\dot{F}^{\beta}_{\vp,q}(\va),\cD)} = \|f\|_{\dot{F}^{\beta}_{\vp,q}(\va)} +
  |f|_{\mathcal{A}^\alpha_q(\dot{F}^{\beta}_{\vp,q}(\va),\cD)}$, for $0<q,\alpha<\infty$. When
  $q=\infty$, the $\ell_q$ norm is replaced by the sup-norm.
\end{mydef}

 Now a fundamental question is whether it is possible to completely characterize $\mathcal{A}^\alpha_q(\dot{F}^{\beta}_{\vp,q}(\va),\cD)$ in terms of known smoothness spaces. Often this problem is addressed by studying Jackson and Bernstein estimates for the dictionary $\cD$. Suppose $Y\hookrightarrow \dot{F}^{\beta}_{\vp,q}(\va) $ is a quasi-normed linear subspace.
Then a Jackson inequality with exponent $r_J>0$ is a so-called direct estimate of the type
$$\sigma_m(f,\cD)_{\dot{F}^{s}_{\vp,q}(\va)}\leq C m^{-r_J}\|f\|_Y,\qquad m\in\bN, f\in Y,$$
 while a Bernstein estimate with exponent $r_B>0$ is an inverse estimate of the type
 $$\|S\|_Y\leq Cm^{r_B} \|S\|_{\dot{F}^{s}_{\vp,q}(\va)},\qquad  S\in \Sigma_m(\cD).$$
 Whenever it is possible to find an appropriate subspace $Y$ with corresponding Jackson and Bernstein estimates with matching exponents, one can call on general theory to obtain a characterisation of $\mathcal{A}^\alpha_q(\dot{F}^{\beta}_{\vp,q}(\va),\cD)$ in terms of certain interpolation spaces between $Y$ and $\dot{F}^{\beta}_{\vp,q}(\va),\cD)$, we refer to \cite[Chap.\ 7]{DeVore1993} for further details.
 
 A general challenge is to obtain good estimates of $\sigma_m(f,\cD)_{\dot{F}^{s}_{\vp,q}(\va)}$ as $m\rightarrow \infty$. Sometimes this can be accomplished by using a simple thresholding approach.  We will need some additional terminology to study this question further.

For $\Gamma\subset \bN$, put $\Sigma_\Gamma:=\sum_{j\in \Gamma} g_j$. We say that $\cD$ is {\em democratic} if there exists $c\geq 1$ such that for all finite sets $\Gamma,\Gamma'\subset \bN$ with $\#\Gamma=\#\Gamma'$,
$$c^{-1} \|\Sigma_\Gamma\|_{\dot{F}^{\beta}_{\vp,q}(\va)}\leq \|\Sigma_{\Gamma'}\|_{\dot{F}^{\beta}_{\vp,q}(\va)}\leq c  \|\Sigma_\Gamma\|_{\dot{F}^{\beta}_{\vp,q}(\va)}.$$

Let us add some more structure to make thresholding  well defined. Let us suppose $\cD\subset L_2(\bR^d)\cap \dot{F}^{\beta}_{\vp,q}(\va)$ is such that there exists an associated dual system $\{\tilde{g}_j\}\subset L_2(\bR^d)$ satisfying
$$\langle \tilde{g}_i,g_j\rangle = \delta_{i,j},\qquad i,j\in\bN.$$
A greedy algorithm of step $N$ is defined as a correspondence
$$\mathcal{G}_N(f):=\sum_{\ell=1}^N \langle f,\tilde{g}_{n_\ell}\rangle g_{n_\ell},$$
for any permutation $\{n_i\}_{i=1}^\infty$ of $\bN$ such that
$$\|\langle f,\tilde{g}_{n_1}\rangle g_1\|_{\dot{F}^{\beta}_{\vp,q}(\va)}\geq \|\langle f,\tilde{g}_{n_2}\rangle g_2\|_{\dot{F}^{\beta}_{\vp,q}(\va)}\geq \|\langle f,\tilde{g}_{n_3}\rangle g_3\|_{\dot{F}^{\beta}_{\vp,q}(\va)}\geq  \cdots$$
We say that $\cD$ is {\em greedy} provided there exists a constant $C\geq 1$ such that
$$\|f-\mathcal{G}_m(f)\|_{\dot{F}^{\beta}_{\vp,q}(\va)}\leq C\sigma_m(f,\cD)_{\dot{F}^{s}_{\vp,q}(\va)},\qquad m\in \bN,$$
for all $f\in \dot{F}^{\beta}_{\vp,q}(\va)$.

Let us turn our attention to the system $\cW$. We have the following first observation regarding the basis $\cW$ and greediness.

\begin{mylemma}\label{le:dem}
Let $\vp=(p_1,p_2,\ldots,p_d)\in (0,\infty)^d$, $0<q<\infty$ and $\beta\in\bR$. Let $\cW_{\vp}$ be a re-scaled quasi-normalized version of $\cW$ in  $\dot{F}^{\beta}_{\vp,q}(\va)$. Then $\cW_{\vp}$ is greedy in $\dot{F}^{\beta}_{\vp,q}(\va)$ if and only if $\cW_{\vp}$ is democratic in $\dot{F}^{\beta}_{\vp,q}(\va)$.
\end{mylemma}
\begin{proof}
 A general result by Temlyakov and Konyagin, see \cite{Konyagin1999} and \cite{Garrigos2004}, yields that $\cW_{\vp}$ is greedy if and only if $\cW_{\vp}$ is democratic and forms an unconditional basis for $\dot{F}^{\beta}_{\vp,q}(\va)$. We have from Theorem \ref{prop:modu} that $\cW_{\vp}$ forms an unconditional basis for $\dot{F}^{\beta}_{\vp,q}(\va)$, so  $\cW_{\vp}$ is greedy if and only if it is democratic.
\end{proof}

\begin{myre}
Temlyakov and Konyagin \cite{Konyagin1999}   proved their result in a Banach space setting, but the reader can easily verify  that their arguments also work in a quasi-Banach setting as the one used in Lemma \ref{le:dem}. This generalisation has also been observed in    \cite{Garrigos2004}.
\end{myre}

We can now characterise exactly when  $\cW_{\vp}$ is democratic in   $\dot{F}^{\beta}_{\vp,q}(\va)$.

\begin{myprop}\label{prop:demo}
Let $\vp=(p_1,p_2,\ldots,p_d)\in (0,\infty)^d$, $d\geq 2$. Let $\cW_{\vp}$ be a re-scaled quasi-normalized version of $\cW$ in  $\dot{F}^{\beta}_{\vp,q}(\va)$. Then $\cW_{\vp}$ is democratic in $\dot{F}^{\beta}_{\vp,q}(\va)$ if and only if  $p_1=p_2=\cdots=p_d$.
\end{myprop}
\begin{proof}
By calling on the characterization given in Theorem \ref{prop:modu}, we notice that quasi-normalization of $\cW$ in $\dot{F}^{\beta}_{\vp,q}(\va)$ implies that 
$$\cW_{\vp}=\{c_{n,R_{j,k}}w_{n,R_{j,k}}:n\in\bN_0^d,j\in\bZ,k\in E\},$$
with normalisation constants $\{c_{n,R_{j,k}}\}$ satisfying
$$c_{n,R_{j,k}}\asymp 2^{j\big(\frac{a_1}{p_1}+\cdots+\frac{a_d}{p_d}-\beta\big)},$$
uniformly in $n,j$ and $k$, with the exponent $\beta-\big(\frac{a_1}{p_1}+\cdots+\frac{a_d}{p_d}\big)$ representing the  differential dimension of  $\dot{F}^{\beta}_{\vp,q}(\va)$. Let us suppose $\cW_{\vp}$ is democratic in $\dot{F}^{\beta}_{\vp,q}(\va)$. For $m,n\in \{1,2,\ldots,d\}$, with $m\not=n$, and $N>1$, put 
\begin{equation}\label{eq:FN}
F_N(x):=\sum_{\ell=1}^N c_{\ell e_n,R_{0,k}}w_{\ell e_n,R_{0,k}}, \qquad G_N(x):=\sum_{\ell=1}^N c_{\ell e_m,R_{0,k}}w_{\ell e_m,R_{0,k}},
\end{equation}
where $e_i$ is the $i$'th basis vector from the canonical basis for $\bR^d$, and  $k\in E$ is fixed. We are supposing that $\cW_{\vp}$ is democratic, so
$$ \|F_N\|_{\dot{F}^{\beta}_{\vp,q}(\va)}\asymp \|G_N\|_{\dot{F}^{\beta}_{\vp,q}(\va)}$$
uniformly in $N\in \bN$. However, using the characterization given in Theorem \ref{prop:modu}, we see that
$$\|F_N\|_{\dot{F}^{\beta}_{\vp,q}(\va)}\asymp
\bigg\|\sum_{\ell=1}^N \mathds{1}_{U(R_{0,k},\ell e_n)}(\cdot)\bigg \|_{\vp}\asymp
 N^{1/p_n},$$ 
 and
 $$\|G_N\|_{\dot{F}^{\beta}_{\vp,q}(\va)}\asymp \bigg\|\sum_{\ell=1}^N \mathds{1}_{U(R_{0,k},\ell e_m)}(\cdot)\bigg \|_{\vp}\asymp N^{1/p_m},$$
so we obtain $ N^{1/p_n}\asymp  N^{1/p_m}$ and conclude that $p_n=p_m$ by letting $N\rightarrow \infty$. Hence, the only case where the $\cW_{\vp}$ can possibly  be democratic is the unmixed  case $p_1=p_2=\cdots=p_d$.

Now suppose $p_1=p_2=\cdots=p_d=p$ for some $p\in(0,\infty)$. Calling on the characterization given in Theorem \ref{prop:modu}, democracy of $\cW_{(p,p,\ldots,p)}$ will follow if we can  prove that the special sequence,
$$\tilde{\mathsf{1}}_\Gamma:=\sum_{F\in \Gamma} \frac{\mathds{1}_F}{\|\mathds{1}_F\|_{\dot{f}^\beta_{(p,p,\ldots,p),q}(\va)}},$$
where $\mathds{1}_F$ denotes the sequence defined on $\cR\times \bN_0^d$ with entry $1$ at $F$ and zero otherwise, satisfies
\begin{equation}\label{eq:demo}
    \|\tilde{\mathsf{1}}_\Gamma\|_{\dot{f}^{\beta}_{(p,p,\ldots,p),q}(\va)}\asymp (\# \Gamma)^{1/p},
\end{equation}
uniformly, for every finite subset $\Gamma\subset \cR\times \bN_0^d$.

However, this particular sequence space estimate in the unmixed setting has been studied in detail by Garrig\'os and Hern\'andez in \cite{Garrigos2004}, where a minor modification to the proof of  \cite[Prop.\ 3.2]{Garrigos2004}, in order to adapt  to \eqref{eq:redef}, confirms the estimate  $\eqref{eq:demo}$.
\end{proof}

We have the following corollary.

\begin{mycor}
Let $\vp=(p_1,p_2,\ldots,p_d)\in (0,\infty)^d$, $d\geq 2$. Let $\cW_{\vp}$ be a re-scaled quasi-normalized version of $\cW$ in  $\dot{F}^{\beta}_{\vp,q}(\va)$. Then $\cW_{\vp}$ is greedy in $\dot{F}^{\beta}_{\vp,q}(\va)$ if and only if  $p_1=p_2=\cdots=p_d$.
\end{mycor}
\begin{proof}
Clearly, by the characterization given in Theorem \ref{prop:modu}, $\cW_{\vp}$ is greedy in $\dot{F}^{\beta}_{\vp,q}(\va)$ if and only if 
$\{\mathds{1}_F\}_{F\in  \cR\times \bN_0^d}$ is greedy in $\dot{f}^{\beta}_{\vp,q}(\va)$. We notice that $\dot{f}^{\beta}_{\vp,q}(\va)$ satisfies the hypothesis of  \cite[Theorem 2.1]{Garrigos2004}, and invoking \cite[Theorem 2.1]{Garrigos2004}, we obtain that $\cW_{\vp}$ is greedy in $\dot{F}^{\beta}_{\vp,q}(\va)$ if and only if  
$$\bigg\{\frac{\mathds{1}_F}{\|\mathds{1}_F\|_{\dot{f}^\beta_{\vp,q}(\va)}}\bigg\}_{F\in  \cR\times \bN_0^d}$$
is democratic in $\dot{f}^{\beta}_{\vp,q}(\va)$. As noted in (the proof of) Proposition \ref{prop:demo}, this happens if and only if $p_1=p_2=\cdots=p_d$.
\end{proof}

\subsection{A mixed-norm Bernstein estimate}
Now we consider a Bernstin estimate for the system $\cW$, where we choose a Besov space to be the embedded smoothness space. 
The technique we employ below to obtain the estimate for $\cW$ is to "unmix" the problem by using the  embedding result given in the following Lemma \ref{lem:emb}, which is proven in \cite[Proposition 3.4]{Cleanthous2019}.

\begin{mylemma}\label{lem:emb}
Let $s,t\in\bR,\;\vec{p}=(p_1,\cdots,p_d)\in(0,\infty)^d,\;q\in(0,\infty]$ and $\vec{a}\in [1,\infty)^n.$ We set $p_{\text{min}}:=\min(p_1,\dots,p_d)$ and 
$p_{\text{max}}:=\max(p_1,\dots,p_d),$ then
$$\dot{B}^t_{p_{\text{min}}q}(\vec{a})\hookrightarrow\dot{B}^s_{\vec{p}q}(\vec{a})\hookrightarrow\dot{B}^\tau_{p_{\text{max}}q}(\vec{a})\;\;\text{and}\;\;\dot{F}^t_{p_{\text{min}}q}(\vec{a})\hookrightarrow\dot{F}^s_{\vec{p}q}(\vec{a})\hookrightarrow\dot{F}^\tau_{p_{\text{max}}q}(\vec{a}),$$
where
$$t=s-\Big(\frac{a_1}{p_1}+\cdots+\frac{a_d}{p_d}\Big)+\frac{\nu}{p_{\text{min}}}\:\:\text{and}\;\;\tau=s-\Big(\frac{a_1}{p_1}+\cdots+\frac{a_d}{p_d}\Big)+\frac{\nu}{p_{\text{max}}}.$$
\end{mylemma}

We have the following Bernstein inequality for the system $\cW$.
\begin{myprop}\label{prop:bern}
Let $\vp,\vec{\tau}\in (0,\infty)^d$, with $\tau_i<p_i$, $i=1,\ldots,d$, where we also suppose 
$\tau_{\text{min}}:=\min\{\tau_i\}_{i=1}^d<p_{\text{max}}:=\max\{p_i\}_{i=1}^d$,
and let  $0<q<\infty$. Then there exists a constant $C$ such that for any $m\geq 1$, $0<r\leq \infty$, and $\alpha>0$,
$$\|S\|_{\dot{B}^{\alpha}_{\vec{\tau},q}(\va)}\leq Cm^{1/\tau_{\text{min}}-1/p_{\text{max}}} \|S\|_{\dot{F}^{\beta}_{\vp,r}(\va)},\qquad S\in \Sigma_m(\cW),
$$
with
$$\alpha-\beta=\sum_{j=1}^d \frac{a_i}{\tau_i}-\sum_{j=1}^d \frac{a_i}{p_i},$$
and $\tau_{\text{min}}=\min\{\tau_i\}_{i=1}^d$ and $p_{\text{max}}=\max\{p_i\}_{i=1}^d$. Moreover, the exponent $1/\tau_{\text{min}}-1/p_{\text{max}}$ is the best possible.
\end{myprop}
\begin{proof}
Let $S\in \Sigma_m(\cW)$. Using the (lower) Besov embedding from Lemma \ref{lem:emb}, with
\begin{equation}\label{eq:bern1}
    t=\alpha- \sum_{j=1}^d \frac{a_j}{\tau_j}+\frac{\nu}{\tau_{\text{min}}},
\end{equation}
we obtain
\begin{align*}
\|S\|_{\dot{B}^{\alpha}_{\vec{\tau},q}(\va)}&\lesssim \|S\|_{\dot{B}^{t}_{\tau_{\text{min}},q}(\va)}\\&\lesssim
\|\{\langle S,w_{n,R}\rangle\}_{n,R} \|_{\dot{b}^{t}_{\tau_{\text{min}},q}(\va)}.
\end{align*}
We continue the estimate, where we call on the known Bernstein estimate for sequence spaces  in the unmixed setting \cite[Theorem 5.4]{Garrigos2004}, with parameters
\begin{equation}\label{eq:bern2}s=t+\frac{\nu}{p_{\text{max}}}-\frac{\nu}{\tau_{\text{min}}}.
\end{equation}
We obtain, for $0<r\leq \infty$,
\begin{align*}
%&\lesssim\|\{\langle S,w_{n,R}\rangle\}_{n,R} \|_{\dot{b}^{t}_{p_{\text{max}},q}(\va)}\\
\|\{\langle S,w_{n,R}\rangle\}_{n,R} \|_{\dot{b}^{t}_{\tau_{\text{min}},q}}&\lesssim m^{1/\tau_{\text{min}}-1/p_{\text{max}}}\|\{\langle S,w_{n,R}\rangle\}_{n,R} \|_{\dot{f}^{s}_{p_{\text{max}},r}(\va)}\\
&\lesssim m^{1/\tau_{\text{min}}-1/p_{\text{max}}}\|\{\langle S,w_{n,R}\rangle\}_{n,R} \|_{\dot{F}^{s}_{p_{\text{max}},r}(\va)}\\
&\lesssim m^{1/\tau_{\text{min}}-1/p_{\text{max}}}\|\{\langle S,w_{n,R}\rangle\}_{n,R} \|_{\dot{F}^{\beta}_{\vp,r}(\va)},
\end{align*}
where the last inequality follows by the upper embedding from Lemma \ref{lem:emb} with
\begin{equation}\label{eq:bern3}\beta=s+\sum_{j=1}^d \frac{a_j}{p_j}-\frac{\nu}{p_{\text{max}}}.
\end{equation}
Hence, we obtain
$$\|S\|_{\dot{B}^{\alpha}_{\vp,q}(\va)}\leq Cm^{1/\tau_{\text{min}}-1/p_{\text{max}}} \|S\|_{\dot{F}^{\beta}_{\vp,r}(\va)},$$
where, solving for $\alpha$ and $\beta$ in the system of equations given by \eqref{eq:bern1},  \eqref{eq:bern2}, and \eqref{eq:bern3},
$$\alpha-\beta=\sum_{j=1}^d \frac{a_i}{\tau_i}-\sum_{j=1}^d \frac{a_i}{p_i}.$$
Finally, to verify that the exponent $1/\tau_{\text{min}}-1/p_{\text{max}}$  is optimal, we consider the scenario where for some $n\in \{1,2,\ldots,d\}$, 
$\tau_{\text{min}}=\tau_n$ and $p_{\text{max}}=p_n$. Then we consider the function $F_N\in \Sigma_N(\cW)$ given in Eq.\ \eqref{eq:FN}, where we use the same calculation as in the proof of Proposition \ref{prop:demo} to deduce that, uniformly in $N$, 
$$\|F_N\|_{\dot{B}^{\alpha}_{\vec{\tau},q}(\va)}\asymp  N^{1/\tau_n}\quad\text{and}\quad \|F_N\|_{\dot{F}^{\beta}_{\vp,q}(\va)}
\asymp
 N^{1/p_n}.$$
 Hence,
 $$\|F_N\|_{\dot{B}^{\alpha}_{\vec{\tau},q}(\va)}\asymp N^{1/\tau_n-1/p_n}\|F_N\|_{\dot{F}^{\beta}_{\vp,q}(\va)},$$
 which proves the claim.
\end{proof}
\begin{myre}
In the unmixed case, we recover a Bernstein estimate similar to the known estimate for wavelet system with the perhaps more familiar looking smoothness relation $(\alpha-\beta)/\nu=1/\tau-1/p$. However, we allow for a completely arbitrary anisotropy $\va$, so even in the unmixed case, Proposition \ref{prop:bern} covers  new cases.  
\end{myre}

\subsection{A mixed-norm Jackson estimate}
We have the following Jackson inequality for the system $\cW$.
\begin{myprop}\label{prop:jack}
Let $\vp,\vec{\tau}\in (0,\infty)^d$, with $\tau_i<p_i$, $i=1,\ldots,d$, where we also suppose
 $\tau_{\text{max}}:=\max\{\tau_i\}_{i=1}^d<p_{\text{min}}:=\min\{p_i\}_{i=1}^d$,
 and let  $0<q<\infty$. Then there exists a constant $C$ such that for any $m\geq 1$, $0<r\leq \infty$, and $\alpha\in \bR$,
$$\sigma_m(f,\cW)_{\dot{F}^{\beta}_{\vp,r}(\va)}\leq Cm^{-(1/\tau_{\text{max}}-1/p_{\text{min}})} \|f\|_{\dot{B}^{\alpha}_{\vec{\tau},\tau}(\va)},\qquad f\in \dot{B}^{\alpha}_{\vec{\tau},r}(\va),
$$
with
$$\alpha-\beta=\sum_{j=1}^d \frac{a_i}{\tau_i}-\sum_{j=1}^d \frac{a_i}{p_i}.$$
% Moreover, the exponent $1/\tau_{\text{max}}-1/p_{\text{min}}$ is the best possible.
\end{myprop}
\begin{myre}
It will actually follow from the proof of Proposition \ref{prop:jack} below that we do have the embedding $\dot{B}^{\alpha}_{\vec{\tau},r}(\va)\hookrightarrow \dot{F}^{\beta}_{\vp,r}(\va)$, so $\sigma_m(f,\cW)_{\dot{F}^{\beta}_{\vp,r}(\va)}$ is well-defined for $f\in \dot{B}^{\alpha}_{\vec{\tau},r}(\va)$ with the specified parameters.
\end{myre}
\begin{proof}[Proof of Proposition \ref{prop:jack}]
Given $f\in \dot{B}^{\beta}_{\vec{\tau},r}(\va)$, $m\in \bN$, and $S_m\in \Sigma_m(\cW)$.
 Using the (lower) Triebel-Lizorkin embedding from Lemma \ref{lem:emb}, with
\begin{equation}\label{eq:jack1}
    s=\beta- \sum_{j=1}^d \frac{a_j}{p_j}+\frac{\nu}{p_{\text{min}}},
\end{equation}
we obtain
\begin{align*}
\|f-S_m\|_{\dot{F}^{\beta}_{\vp,r}(\va)}&\lesssim \|f-S_n\|_{\dot{F}^{s}_{p_{\text{min}},r}(\va)}    \\
&\lesssim 
\|\{\langle f-S_m,w_{n,R}\rangle\}_{n,R} \|_{\dot{f}^{s}_{p_{\text{min}},r}(\va)}.
\end{align*}
We now choose $S_m$ such that 
$$\|\{\langle f-S_m,w_{n,R}\rangle\}_{n,R} \|_{\dot{f}^{s}_{p_{\text{max}},r}(\va)}\leq 2\sigma_m(\{\langle f,w_{n,R}\rangle\}_{n,R})_{\dot{f}^{s}_{p_{\text{min}},r}(\va)}. $$
Hence, with this choice of $S_m$,
$$\|f-S_m\|_{\dot{F}^{\beta}_{\vp,r}(\va)}\lesssim \sigma_m(\{\langle f,w_{n,R}\rangle\}_{n,R})_{\dot{f}^{s}_{p_{\text{min}},r}(\va)}.$$
We now call on a known Jackson estimate for sequence spaces  in the unmixed setting \cite[Theorem 4.3]{Garrigos2004}, with parameters
\begin{equation}\label{eq:jack2} t-s=\frac{\nu}{\tau_{\text{max}}}-\frac{\nu}{p_{\text{min}}}.
\end{equation}
We obtain, for $0<r\leq \infty$,
\begin{align*}
  \sigma_m(\{\langle f,w_{n,R}\rangle\}_{n,R})_{\dot{f}^{s}_{p_{\text{min}},r}(\va)}&\lesssim
  m^{-(1/\tau_{\text{max}}-1/p_{\text{min}})}\|\{\langle f,w_{n,R}\rangle\}_{n,R}\|_{\dot{b}^{t}_{\tau_{\text{max}},r}(\va)}\\
  &\lesssim
  m^{-(1/\tau_{\text{max}}-1/p_{\text{min}})}\|f\|_{\dot{B}^{t}_{\tau_{\text{max}},r}(\va)}\\
   &\lesssim
  m^{-(1/\tau_{\text{max}}-1/p_{\text{min}})}\|f\|_{\dot{B}^{\alpha}_{\vec{\tau},r}(\va)},
\end{align*}
where the last inequality follows from the Besov space embedding in Lemma \ref{lem:emb} with
\begin{equation}\label{eq:jack3}\alpha=s+\sum_{j=1}^d \frac{a_j}{\tau_j}-\frac{\nu}{\tau_{\text{max}}}.
\end{equation}
Hence,
\begin{align*}
\sigma_m(f,\cW)_{\dot{F}^{\beta}_{\vp,r}(\va)}\leq \|f-S_m\|_{\dot{F}^{\beta}_{\vp,r}(\va)}&\lesssim
m^{-(1/\tau_{\text{max}}-1/p_{\text{min}})}\|f\|_{\dot{B}^{\alpha}_{\vec{\tau},r}(\va)}.
\end{align*}
where, solving for $\alpha$ and $\beta$ in the system of equations given by \eqref{eq:jack1},  \eqref{eq:jack2}, and \eqref{eq:jack3},
$$\alpha-\beta=\sum_{j=1}^d \frac{a_i}{\tau_i}-\sum_{j=1}^d \frac{a_i}{p_i}.$$
\end{proof}

\begin{myre}
The fact that $\cW$ cannot be normalized to be greedy in $\dot{F}^{\beta}_{\vp,r}(\va)$, except in the unmixed case,  makes an estimate of the error of best $m$-term approximation more illusive, and at present we do not know whether the Jackson inequality exponent in Proposition \ref{prop:jack} is optimal.
 
 However, we may replace best $m$-term approximation by greedy $m$-term approximation in Proposition \ref{prop:jack}, relying on exactly the same proof. For greedy $m$-term approximation, the Jackson exponent in Proposition \ref{prop:jack} {\em is} optimal. To verify this, consider the scenario where for some $n\in \{1,2,\ldots,d\}$, 
$\tau_{\text{max}}=\tau_n$ and $p_{\text{min}}=p_n$. Then we consider the function $F_N$ given in Eq.\ \eqref{eq:FN}. We notice, for small $\epsilon>0$,
$$(F_{2N}-\epsilon F_N)-\mathcal{G}_N(F_{2N}-\epsilon F_N)=(1-\epsilon) F_N,$$
which implies that
$$\|(F_{2N}-\epsilon F_N)-\mathcal{G}_m(F_{2N}-\epsilon F_N)\|_{\dot{F}^{\beta}_{\vp,r}(\va)}\asymp N^{1/p_n},$$
while
$$\|F_N\|_{\dot{B}^{\alpha}_{\vec{\tau},\tau}(\va)}\asymp N^{1/\tau_n},$$
so
$$\|(F_{2N}-\epsilon F_N)-\mathcal{G}_N(F_{2N}-\epsilon F_N)\|_{\dot{F}^{\beta}_{\vp,r}(\va)}\asymp N^{1/\tau_n-1/p_n}\|F_N\|_{\dot{B}^{\alpha}_{\vec{\tau},\tau}(\va)},$$
and letting $N\rightarrow \infty$  proves optimality.
\end{myre}
\subsection{Embeddings of the mixed-norm approximation spaces}
It is well-known from the general theory of non-linear approximation  that Jackson and Bernstein estimates can be used to obtain embedding results for the approximation spaces $\mathcal{A}^\alpha_q(\dot{F}^{\beta}_{\vp,q}(\va),\cD)$, and in the optimal scenario, where the exponents of the Bernstein and Jackson estimates match, we get a full characterization of $\mathcal{A}^\alpha_q(\dot{F}^{\beta}_{\vp,q}(\va),\cD)$ as a real interpolation space.

However, we notice from Propositions \ref{prop:bern} and \ref{prop:jack} that the exponents in the present setup, in general, {\em do not} match up. In face, the only case where we obtain matching exponents is in the unmixed case. We summarize the embeddings that can be obtained directly from  Propositions \ref{prop:bern} and \ref{prop:jack} in the following Corollary that will conclude this paper, where we refer to \cite{DeVore1993} for a discussion of interpolation spaces, including the real method of interpolation, which is used in the statement of the result.

\begin{mycor}
Let $\vp,\vec{\tau}\in (0,\infty)^d$, with $\tau_i<p_i$, $i=1,\ldots,d$, where we also suppose
 $\tau_{\text{max}}:=\max\{\tau_i\}_{i=1}^d<p_{\text{min}}:=\min\{p_i\}_{i=1}^d$,
 and suppose  $0<q<\infty$. Put $$r_J:=\frac{1}{\tau_{\text{max}}}-\frac{1}{p_{\text{min}}}\quad\text{ and }\quad r_B:=\frac{1}{\tau_{\text{min}}}-\frac{1}{p_{\text{max}}}.$$
For $0<s<r_J$, and $0<t\leq \infty$, we have the Jackson embedding
  $$ \big(\dot{F}^{\beta}_{\vp,q}(\va),\dot{B}^{\alpha}_{\vec{\tau},\tau}(\va)\big)_{s/r_J,t}\hookrightarrow \mathcal{A}^s_t(\dot{F}^{\beta}_{\vp,q}(\va),\cW).$$
 For $0<s<r_B$, and $0<t\leq \infty$, we have the Bernstein embedding
$$
\mathcal{A}^s_t(\dot{F}^{\beta}_{\vp,q}(\va),\cW) \hookrightarrow
\big(\dot{F}^{\beta}_{\vp,q}(\va),\dot{B}^{\alpha}_{\vec{\tau},\tau}(\va)\big)_{s/r_B,t}.$$
\end{mycor}
\begin{proof}
  Follows from Propositions \ref{prop:bern} and \ref{prop:jack}  combined with  standard embedding results, see e.g.\ \cite[Chap.\ 7]{DeVore1993}.
\end{proof}
\appendix
\section{Univariate Brushlet bases}\label{app:brush}
We first introduce brushlets in a univariate setting. 
Each univariate brushlet basis is associated with a partition of the
frequency axis $\bR$. The partition can  be chosen with almost
no restrictions, but in order to have good properties of the
associated basis we need to impose some growth conditions on the
partition. 
We have the following definition.

\begin{mydef}
A family $\ptt$ of intervals is called a {\it disjoint covering} of $\R$ if
it consists of a countable set of pairwise disjoint half-open intervals
$I=[\alpha_I,\alpha'_I)$, $\alpha_I<\alpha'_I$, such
that $\cup_{I\in \ptt}I=\R$. If, furthermore, each interval in $\ptt$
has a unique adjacent interval in $\ptt$ to the left and to the right,
and there exists a constant $A>1$ such that
\begin{equation}\label{eq:growthcond}
A^{-1}\leq \frac{|I|}{|I'|} \leq A,\qquad \text{for
  all adjacent}\; I,I'\in \ptt,
\end{equation}
we call $\ptt$ a {\it moderate disjoint covering} of $\R$.
\end{mydef}

Given a moderate disjoint covering $\ptt$ of $\R$, assign to each interval $I\in \ptt$ a cutoff
radius $\varepsilon_I>0$ at the left endpoint and a cutoff radius
$\varepsilon'_I>0$ at the right endpoint, satisfying
\begin{equation}\label{eq:epsilon}
\begin{cases}
\text{(i)}& \ve'_I= \ve_{I'}\; \text{whenever}\;
\alpha'_I=\alpha_{I'}\\
\text{(ii)}& \ve_I+\ve'_I\leq |I|\\
\text{(iii)}& \ve_I\geq c |I|,
\end{cases}
\end{equation}
with $c>0$ independent of $I$.

We are now ready to define the brushlet system. For each $I\in \ptt$,
we 
will  construct a smooth bell function localized in a
neighborhood of this interval. Take a
non-negative ramp function $\rho \in C^\infty(\R)$ satisfying
\begin{equation}\label{eq:ramp}
\rho(\xi)=\left\{ \begin{array}{ll}
0&\mbox{for}\; \xi \leq -1,\\
1&\mbox{for}\; \xi \geq 1,
\end{array} \right.
\end{equation}
with the property that
\begin{equation}\label{eq:ramp2}
\rho (\xi)^2+\rho(-\xi)^2=1\qquad \text{for all}\; \xi \in\R.
\end{equation}
Define for each $I = [\alpha_I,\alpha'_I)\in \ptt$ the {\it bell function}
\begin{equation}\label{eq:bell}
b_I(\xi) := \rho \bigg( \frac{\xi-\alpha_I}{\ve_I}\bigg)\rho
\bigg( \frac{\alpha'_I-\xi}{\ve'_I}\bigg) .
\end{equation}
Notice that $\supp (b_I) \subset [\alpha_I-\ve_I,\alpha'_I+\ve'_I]$ and
$b_I(\xi)=1$ for $\xi \in [\alpha_I+\ve_I,\alpha'_I-\ve'_I]$.
Now the set of local cosine functions
\begin{equation}\label{eq:brush1}
\hat{w}_{n,I}(\xi) =\sqrt{\frac{2}{|I|}} b_I(\xi)
\cos\biggl( \pi \bigl( n+{\textstyle \frac12}\bigr)
\frac{\xi- \alpha_I}{|I|} \biggr) ,\quad n\in \N_0,\quad I\in \ptt,
\end{equation}
 with $\bN_0:=\bN\cup \{0\}$, constitute an orthonormal basis for $L_2(\bR)$, see e.g.\
\cite{Auscher1992}. The collection $\{ w_{n,I}\colon I\in\ptt, n\in\N_0\}$ is called a
{ brushlet system}. The brushlets also have an explicit
representation in the time domain. Define the set of {central
  bell functions} $\{ g_I\}_{I\in \ptt}$ by
\begin{equation}\label{eq:gb}
\hat{g}_I(\xi) := \rho \biggl(
\frac{|I|}{\ve_I} \xi\biggr) \rho \biggl( \frac{|I|}{\ve'_I} (1-\xi)\biggr),
\end{equation}
such that $b_I(\xi) = \hat{g}_I\bigl(|I|^{-1}(\xi -\alpha_I)\bigr)$,
 and let for notational convenience
$$e_{n,I}:= \frac{\pi \bigl( n+{\textstyle
    \frac12}\bigr)}{|I|},\qquad
I\in \ptt,\; n\in \N_0.$$
Then,
\begin{equation}
w_{n,I}(x)
= \sqrt{\frac{|I|}{2}} e^{i\alpha_Ix} \bigl\{ g_I\bigl(
  |I|(x +e_{n,I}) \bigr) + g_I\bigl( |I|(x -e_{n,I}) \bigr)
  \bigr\}.\label{eq:gw}
\end{equation}

By a straight forward calculation it can be verified, see, e.g.,
\cite{Borup2003}, that for $r\geq 1$ there exists a constant $C:=C(r)<\infty$, independent
of $I\in \ptt$, such that  
\begin{equation}\label{eq:gdecay}
|g_I(x)|\leq C(1+|x|)^{-r}.
\end{equation}  Thus a
brushlet $w_{n,I}$ essentially consists of two well localized humps at
the points $\pm e_{n,I}$.

Given a bell function $b_I$, define an
operator $\mathcal{P}_I:L_2(\bR) \rightarrow L_2(\bR)$ by
\begin{equation}
  \label{eq:proje}
  \widehat{\mathcal{P}_If}(\xi) := b_I(\xi)\bigl[ b_I(\xi)\hat{f}(\xi) +
b_I(2\alpha_I-\xi)\hat{f}(2\alpha_I-\xi)- b_I(2\alpha'_I-\xi)\hat{f}(2\alpha'_I-\xi)\bigr].
\end{equation}
It can be verified that $\mathcal{P}_I$ is an orthogonal projection, mapping
$L_2(\bR)$ onto $$\overline{\Span}\{w_{n,I}\colon n\in \N_0\}.$$  Below we will need some of the finer properties of the
operator given by \eqref{eq:proje}. Let us list properties here, and
refer the reader to \cite[Chap.\ 1]{Hernandez1996} for a more detailed
discussion of the properties of local trigonometric bases.

 Suppose $I = [\alpha_I,\alpha'_I)$
and $J=[\alpha_J,\alpha_J')$ are two adjacent compatible intervals
(i.e., $\alpha'_I=\alpha_J$ and $\epsilon_I'=\epsilon_J$). Then it holds true that
\begin{equation}
  \label{eq:summm}
  \widehat{\mathcal{P}_If}(\xi)+\widehat{\mathcal{P}_Jf}(\xi)=\hat{f}(\xi),\qquad
\xi\in[\alpha_I+\epsilon_I,\alpha_J'-\epsilon_J'],\qquad f\in L_2(\bR).
\end{equation}
One can verify \eqref{eq:summm} using the fact that
$b_I\equiv 1$ on $[\alpha_I+\epsilon_I,\alpha'_I-\epsilon_I']$, and that
$b_J\equiv 1$ on $[\alpha_J+\epsilon_J,\alpha_J'-\epsilon_J']$,
together with the fact that 
$$\text{supp}\big(b_I(\cdot)b_I(2\alpha_I-\cdot)\big)\subseteq
[\alpha_I-\epsilon_I,\alpha_I+\epsilon_I]$$
and
$$\text{supp}\big(b_I(\cdot)b_I(2\alpha_I'-\cdot)\big)\subseteq
[\alpha_I'-\epsilon_I',\alpha_I'+\epsilon_I'].
$$
 For 
$\xi\in[\alpha_I'-\epsilon_I',\alpha_J+\epsilon_J]$ we notice that 
\begin{align}
\widehat{\mathcal{P}_If}(\xi)+&\widehat{\mathcal{P}_Jf}(\xi)=
[b_I^2(\xi)+b_J^2(\xi)(\xi)]\hat{f}(\xi)\notag \\&+
\label{eq:add} b_J(\xi)b_J(2\alpha'-\xi)\hat{f}(2\alpha'-\xi)-
b_I(\xi)b_I(2\alpha'-\xi)\hat{f}(2\alpha'-\xi).  
\end{align}
One can then conclude that \eqref{eq:summm} holds true  using the
following facts (see \cite[Chap.\ 1]{Hernandez1996})
$$b_I(\xi)=b_J(2\alpha_I'-\xi),\qquad b_J(\xi)=b_I(2\alpha_J'-\xi),\qquad
\text{for }\xi\in [\alpha_I'-\epsilon_I',\alpha_J+\epsilon_J],$$
and
$$b_I^2(\xi)+b_J^2(\xi)=1,\qquad \text{for } \xi\in [\alpha_I+\epsilon_I,\alpha_J'-\epsilon_J'].$$
Moreover, 
\begin{equation}\label{eq:proj}\mathcal{P}_I+\mathcal{P}_J=\mathcal{P}_{I\cup J}\end{equation} with the $\epsilon$-values 
$\epsilon_I$ and $\epsilon_J'$ for $I\cup J$.

Now we turn to the multivariate tensor product construction.  For a rectangle $R=I_1\times I_2\times\cdots \times I_d\subset \bR^d$, with 
$I_i = [\alpha_{I_i},\alpha'_{I_i})$, we define
 $$P_R=\bigotimes_{j=1}^d \mathcal{P}_{I_i}.$$ Clearly, $P_R$
 is a projection operator $$P_R\colon L_2(\bR^d) \rightarrow
 \overline{\operatorname{span}}\bigg\{\bigotimes_{j=1}^d w_{i_j,I_j}\colon i_j\in \bN_0\bigg\}.$$
Notice that, 
\begin{equation}
  \label{eq:operatorP}
  P_R = b_R(D)\Bigl[\bigotimes_{j=1}^d
(\operatorname{Id}+R_{\alpha_{I_j}}-R_{\alpha_{I_j}'})
\Bigr]b_R(D),
\end{equation}
where 
\begin{equation}\label{eq:S}
\widehat{b_R(D)f} := b_R\hat{f},
\end{equation}
with  $b_R:=\bigotimes_{j=1}^d b_{I_j}$, and $R_af(x) := e^{i2a}f(-x)$, 
$x,a\in \bR$. The corresponding orthonormal tensor product basis of brushlets is given by
\begin{equation}\label{eq:multivariate}
w_{n,R}:=\bigotimes_{j=1}^d w_{n_j,I_j},\qquad n=(n_1,\ldots, n_d)\in \bN_0^d.    
\end{equation}

Let $\Delta=\text{diag}(|I_i|,|I_2|,\ldots,|I_d|)$. Repeated use of \eqref{eq:gw} yields
\begin{align}
|w_{n,R}(x)|&\leq  2^{-d}|R|^{1/2}\prod_{i=1}^d  \bigg(|g_{I_i}\bigl(
  |I_i|(x_i +e_{n_i,I_i}) \bigr)|+ |g_{I_i}\bigl( |I_i|(x_i -e_{n_i,I_{i}})|\bigr) 
  \bigg)\nonumber\\
  &\leq 2^{-d}|R|^{1/2}\sum_{m=1}^{2^d} |G_R(\Delta(x+U_m e_{n,R}))|,\label{eq:G}
  \end{align}
where $G_R:=g_{I_1}\otimes\cdots\otimes g_{I_d}$, $e_{n,R}:=[e_{n_1,I_1},\ldots,e_{n_d,I_d}]^T$, and
\begin{equation}\label{eq:um}
    U_m:=\text{diag}(v_m),
\end{equation}
with $v_m\in\R^d$ chosen such that $\cup_{m=1}^{2^d}v_m= \{-1,1\}^d$. We also notice that for any $r>0$, using \eqref{eq:gdecay} and \eqref{ad8},
\begin{equation}\label{eq:ngdecay}
|G_R(x)|\leq C_r\brac{x}^{-r},\qquad x\in\R^d,    
\end{equation}
with $C_r$ independent of $R$.

\section{Some technical  proofs}\label{s:app}
We first give a proof of Proposition \ref{prop:onb}.
\begin{proof}[Proof of Proposition \ref{prop:onb}]
Let us consider orthonormality. Let $\mathcal{S}_j:=\{w_{n,R_{j,k}}:k\in E,n\in\bN_0^d\}$ and notice that the functions in $\mathcal{S}_m$ and $\mathcal{S}_n$ have disjoint frequency support for $|m-n|>1$. 
Also notice, using the separable structure of the multi-variate brushlet functions,  that  the partitioning of the sets $K_j$ ensures that $\mathcal{S}_m$ is orthogonal to $\mathcal{S}_{m\pm 1}$, where it is essential that we have chosen compatible $\epsilon$-values for the univariate systems at levels $j=m$ and $j=m\pm 1$, see the specification following \eqref{eq:part}. Within each $\mathcal{S}_j$, orthonormality follows directly from the separable structure of the bi-variate brushlet system and the orthogonality of the respective univariate brushlet systems. 

We will now verify completeness of the system, where  we  first notice that for $f\in L_2(\bR^d)$, 
$$\sum_{R\in \mathcal{P}_{j,k}, j\in \bZ, k\in E} \widehat{P_R f}(\xi),$$
where we use the notation introduced in Appendix \ref{app:brush},
is well-defined and converges pointwise as the support of each term  in the sum overlaps with at most $2^d$ other terms corresponding to adjacent rectangles.  
We will exploit the tensor product structure of $P_R$, where we first sum "too many" terms in the sense that we consider the product set $\tilde{E}:=E_2^d\supsetneq E$, where $E_2:=\{\pm 2,\pm 1\}$ is defined as before.
 We have,
\begin{align*}\sum_{k\in\tilde{E}} P_{R_{j,k}}&=
\sum_{k\in\tilde{E}} P_{I^1_{j,k_1}}\otimes P_{I^2_{j,k_2}}\otimes\cdots\otimes P_{I_{j,k_d}^n}\\&=
\bigg(\sum_{k_1\in E_2} P_{I^1_{j,k_1}}\bigg)\otimes \bigg(\sum_{k_2\in E_2} P_{I^2_{j,k_2}}\bigg)\otimes\cdots\otimes \bigg(\sum_{k_d\in E_2} P_{I^n_{j,k_d}}\bigg)\\&=
P_{I_{j}^1}\otimes P_{I_{j}^2}\otimes\cdots\otimes P_{I_{j}^d},
\end{align*}
where $I_j^i:=\big[-2^{ja_i},2^{ja_i}\big)$ with cutoff radius 
 $2^{(j-2)a_i}$ at $\pm 2^{ja_i}$. In a similar fashion, we deduce that
 $$\sum_{k\in \{\pm 1\}^d} P_{R_{j,k}}=P_{I_{j-1}^1}\otimes P_{I_{j-1}^2}\otimes\cdots\otimes P_{I_{j-1}^d}.$$
 Hence,
 \begin{align}\sum_{k\in{E}} P_{R_{j,k}}&=\sum_{k\in\tilde{E}} P_{R_{j,k}}-\sum_{k\in\{\pm 1\}^d} P_{R_{j,k}}\nonumber\\
 &=P_{I_{j}^1}\otimes P_{I_{j}^2}\otimes\cdots\otimes P_{I_{j}^n}-P_{I_{j-1}^1}\otimes P_{I_{j-1}^2}\otimes\cdots\otimes P_{I_{j-1}^d}.\label{eq:tele}
 \end{align}
 Noticing the telescoping structure of \eqref{eq:tele}, we obtain for $j_0\in\bZ$ and $N\in \bN$,
 \begin{align*}\sum_{j=j_0-N}^{j_0+N}\sum_{k\in{E}} P_{R_{j,k}}=
P_{I_{j_0+N}^1}\otimes P_{I_{j_0+N}^2}\otimes&\cdots\otimes P_{I_{j_0+N}^d}\\&-P_{I_{j_0-N}^1}\otimes P_{I_{j_0-N}^2}\otimes\cdots\otimes P_{I_{j_0-N}^d}.
\end{align*}
% $$\sum_{Q\in \mathcal{A}_j^T} P_Q=P_{A_{j-1}} \otimes P_{I_{j,N_j}},\qquad \sum_{Q\in \mathcal{A}_j^B} P_Q=P_{A_{j-1}} \otimes P_{I_{j,-N_j}},$$
% So, in particular, again using \eqref{eq:proj},
% $$P_{A_{j-1}\times A_{j-1}}+\sum_{Q\in \mathcal{A}_j^T\cup \mathcal{A}_j^T} P_Q=P_{A_{j-1}\times A_{j}}.$$
% Using a similar argument for $ \mathcal{A}_j^L\cup \mathcal{A}_j^R$, and collecting terms,  we may conclude that
% $$P_{A_{j-1}\times A_{j-1}}+\sum_{Q\in \mathcal{A}_j} P_Q=P_{A_{j}\times A_{j}}, \qquad j>1.$$
% Using a parallel argument, we may conclude that
% $$P_{A_{j-1}\times A_{j-1}}+\sum_{Q\in \mathcal{A}_j} P_Q=P_{A_{j}\times A_{j}}, \qquad j\leq -1.$$
% Noting that $A_{-1}=A_1$, it follows that
% $$\sum_{-N \leq j \leq N}\sum_{Q\in \mathcal{A}_j} P_Q=P_{A_{N}\times A_{N}}-P_{A_{-N-1}\times A_{-N-1}}.$$
It thus follows easily  that 
$$\lim_{N\rightarrow \infty}\sum_{j=j_0-N}^{j_0+N}\sum_{k\in{E}} P_{R_{j,k}}=Id_{L_2(\bR^d)},$$
in the strong operator topology. This completes the proof.
\end{proof}

We now give a proof of Lemma \ref{lem:maxbound}, which  was used in the proof of Proposition \ref{lem:reconstruct}.

\begin{proof}[Proof of Lemma \ref{lem:maxbound}]
  From \eqref{eq:ngdecay} we have that 
  \begin{equation}\label{eq:west}|w_{n,R}(x)|\leq
  C_N|R|^{1/2}\sum_{m=1}^{2^d}\brac{
U_m 2^{j\vec{a}}x-\pi(\cn+\ca)}^{-N},
\end{equation} 
for any $N>0$, with $C_N$ independent of $R$.
 Fix $N>d/r$. We can,
  without loss of generality, suppose $x\in
  U(R,{0})$. 
  %Let $A_0=\{\cn\in \bN_0^2\colon |{\TS  \pi(\cn+\ca)|\leq
%    1\}$, and 
    For $j\in \bN$, we let $$A_j=\{ \cn\in \bN_0^d\colon 2^{j-1}<| \pi(\cn+\ca)|_{\va}\leq 2^j\},$$
    and let $A_0:=\{ \cn\in \bN_0^d\colon  \pi(\cn+\ca)|_{\va}\leq  1\}$.
    Notice that $\cup_{\cn\in A_j}
    U(R,\cn)$ is a bounded set contained in the ball
    ${B}_{\va}(0,c2^{j+1}|R|^{-1/2})$. Now, since $|U(R,n)|\asymp |R|^{-1}$ uniformly, and $0<r\leq 1$, and using \eqref{bracket},
    \begin{align*}
      \sum_{\cn\in A_j} &|s_{n,R}|\brac{ 
2^{j\vec{a}}x-\pi(\cn+\ca)}^{-N} \\&\leq C2^{-jN} \sum_{\cn\in
        A_j} |s_{n,R}|\\
&      \leq C2^{-jN} \Bigl( \sum_{\cn\in
        A_j} |s_{n,R}|^r\Bigr)^{1/r}\\
      &\leq C2^{-jN} |R|^{1/r} \biggl( \int \sum_{\cn\in
        A_j} |s_{n,R}|^r\mathds{1}_{U(R,\cn)}(y)\, dy \biggr)^{1/r}\\
      &\leq CL^{1-r}2^{-jN} |R|^{1/r} \biggl( 
%      \frac{2^{(j+1)\nu}t_k^{-\nu}}{2^{(j+1)\nu}t_k^{-\nu}}
\int_{{B}_{\va}(0,c2^{j+1}|R|^{-1/2})} \Bigl(\sum_{\cn\in
        A_j} |s_{n,R}|\mathds{1}_{U(R,\cn)}(y)\Bigr)^r \, dy\biggr)^{1/r}\\
      &\leq C'2^{-j(N-2/r)} \mathcal{M}_r \Bigl(\sum_{\cn\in
       \bN_0^2} |s_{n,R}|\mathds{1}_{U(R,\cn)} \Bigr)(x).
    \end{align*}
We now perform the summation over $j\in \bN_0$ to obtain
$$    \sum_{\cn\in \bN_0^2} |s_{n,R}|\brac{ 
2^{j\vec{a}}x-\pi(\cn+\ca)}^{-N} 
\leq  C\mathcal{M}_r \Bigl(\sum_{\cn\in
       \bN_0^2} |s_{n,R}|\mathds{1}_{U(R,\cn)} \Bigr)(x).$$
We then use the substitutions $x=R_m z$, $m=1,\ldots,2^d$, to cover all $2^d$ terms on the right-hand side of \eqref{eq:west}, where we use that $U_{m}$ and $2^{j\vec{a}}\cdot$ commute. 
\end{proof}

\appendix

%\bibliographystyle{abbrv}
%\bibliography{combined}

\end{document}